\documentclass[11pt,notitlepage,a4paper]{article}

\usepackage{amssymb} 
\usepackage[intlimits]{amsmath}
\usepackage{amsfonts}
\usepackage{amsthm} 

\usepackage{mathptmx}

\usepackage{hyperref}
\hypersetup{colorlinks=true,linkcolor=RoyalPurple,citecolor=RoyalPurple}

\usepackage{cite}

\usepackage{graphicx}

\usepackage{geometry}

\usepackage{color}
\usepackage[dvipsnames]{xcolor}

\let\oldsqrt\sqrt
\def\sqrt{\mathpalette\DHLhksqrt}
\def\DHLhksqrt#1#2{%
\setbox0=\hbox{$#1\oldsqrt{#2\,}$}\dimen0=\ht0
\advance\dimen0-0.2\ht0
\setbox2=\hbox{\vrule height\ht0 depth -\dimen0}%
{\box0\lower0.4pt\box2}}

\allowdisplaybreaks

\newcommand{\R}{\mathbb{R}} 
\newcommand{\N}{\mathbb{N}} 

\newcommand{\dist}{\textnormal{dist}} 

\renewcommand{\phi}{\varphi}

\newcommand{\m}{m}

\newcommand{\cC}{{\mathcal C}}

\newcommand{\cE}{{\mathcal E}}
\newcommand{\cF}{{\mathcal F}}

\newcommand{\cH}{{\mathcal H}}

\newcommand{\cL}{{\mathcal L}}

\newcommand{\Ds}{{(-\Delta)}^s}

\newcommand{\eps}{\varepsilon}

\theoremstyle{definition}
\newtheorem{defi}{Definition}[section]

\theoremstyle{plain} 
\newtheorem{thm}[defi]{Theorem}
\newtheorem{prop}[defi]{Proposition}
\newtheorem{lemma}[defi]{Lemma}
\newtheorem{cor}[defi]{Corollary}

\theoremstyle{definition}

\numberwithin{equation}{section} 

\title{
Positive powers of the Laplacian: \\ from hypersingular integrals to boundary value problems
}
\author{
Nicola Abatangelo\footnote{D\'{e}partement de math\'{e}matique, Universit\'{e} Libre de Bruxelles CP 214, Boulevard du Triomphe, 1050 Ixelles, Belgium, nicola.abatangelo@ulb.ac.be}, 
\ Sven Jarohs\footnote{Institut f\"ur Mathematik, Goethe-Universit\"at, Robert-Mayer-Stra\ss e 10, 60054 Frankfurt, Germany, jarohs@math.uni-frankfurt.de}, and
Alberto Salda\~{n}a\footnote{Institut f\"ur Analysis, Karlsruhe Institute for Technology, Englerstra\ss e 2, 76131, Karlsruhe, Germany, \newline alberto.saldana@partner.kit.edu}
}
\date{}

\begin{document}

\maketitle

\begin{abstract}
\noindent Any positive power of the Laplacian is related via its Fourier symbol to a
hypersingular integral with finite differences. We show how this yields a pointwise evaluation which is more flexible than other notions 
used so far in the literature for powers larger than 1; in particular, this evaluation can be applied to more general boundary value problems
and we exhibit explicit examples. We also provide a natural variational framework and, using an asymptotic analysis, 
we prove how these hypersingular integrals reduce to polyharmonic
operators in some cases. Our presentation aims to be as self-contained as possible and relies on
elementary pointwise calculations and known identities for special functions.
\end{abstract}

\section{Introduction}

Any positive power $s>0$ of the (minus) Laplacian, \emph{i.e.} $(-\Delta)^s$, has the same Fourier symbol (see \cite[Chapter 5]{SKM93} or Theorem \ref{main:thm} below) as the following hypersingular integral,
\begin{align}\label{Ds:def}
L_{m,s} u(x):=\frac{c_{N,\m,s}}{2}\int_{\R^N} \frac{\delta_\m u(x,y)}{|y|^{N+2s}} \ dy,\qquad x\in \R^N,
\end{align}
where $N\in\N$ is the dimension, $m\in\N$, $s\in(0,m)$,
\begin{align*}
\delta_\m u(x,y):= \sum_{k=-\m}^\m (-1)^k { \binom{2m}{m-k}} u(x+ky) \qquad \text{ for }x,y\in \R^N
\end{align*}
is a finite difference of order $2m$, and $c_{N,\m,s}$ is a \emph{positive} constant given by
\begin{align}\label{cNms:def}
c_{N,\m,s}
:=\left\{\begin{aligned}
&\frac{4^{s}\Gamma(\frac{N}{2}+s)}{ \pi^{\frac{N}{2}}\Gamma(-s) }\Big(\sum_{k=1}^{\m}(-1)^{k}{ \binom{2m}{m-k}} k^{2s}\Big)^{-1},&& s\in(0,m)\backslash \N,\\
&\frac{4^{s}\Gamma(\frac{N}{2}+s)s!}{2\pi^{\frac{N}{2}}} \Big(\sum_{k=2}^m (-1)^{k-s+1}{\binom{2m}{m-k}} k^{2s} \ln(k)\Big)^{-1},&& s\in\{1,\ldots,m-1\}.
\end{aligned}\right.
\end{align}
The operator \eqref{Ds:def} is well defined for sufficiently smooth functions satisfying a growth condition at infinity such as $u\in \cL^1_s$, where 
\begin{equation*}
 \cL^1_{s}:=\left\{u\in L^1_{loc}(\R^N)\;:\; \|u\|_{\cL^1_{s}}<\infty \right\}, \qquad \|u\|_{\cL^1_{s}}:=\int_{\R^N}\frac{|u(x)|}{1+|x|^{N+2s}}\ dx\qquad \text{ for }s>0.
\end{equation*}
 
If $s\in(0,1)$, then
\begin{align}\label{frac-lapl}
L_{1,s}u(x)=(-\Delta)^s u(x):=-\frac{4^{s}\Gamma(\frac{N}{2}+s)}{2\pi^{\frac{N}{2}}\Gamma(-s)}\int_{\R^N} \frac{2u(x)-u(x+y)-u(x-y)}{|y|^{N+2s}}\ dy,\qquad x\in\R^N.
\end{align}
The operator $L_{1,s}$ is usually called the \textit{fractional Laplacian}, and it has has attracted much attention from the PDE perspective in recent years, see \cite{BV15} and the references therein.  

If $s=1$, then \eqref{Ds:def} reduces pointwisely to the usual Laplacian, as a matter of fact, 
 \begin{align*}
-\sum_{i=1}^N\partial_{ii}\ u(x)=\frac{\Gamma(\frac{N}{2}+1)}{4\pi^{\frac{N}{2}}\ln (2)}\int_{\R^N}\frac{u(x+2y)-4u(x+y)+6u(x)-4u(x-y)+u(x-2y)}{|y|^{N+2}}\;dy.
 \end{align*}
Actually, as shown in Theorem \ref{van:cor} below, for any $n,m\in\N$, $L_{m+n,n}$ is the polyharmonic operator $(-\Delta)^n u:=(-\sum_{i=1}^{N}\partial_{ii} )^{n}u$ whenever $u$ is $2n$-times continuously differentiable and belongs to $\cL^1_n$. We refer to \cite{GGS10} for a survey on boundary value problems associated to $(-\Delta)^n$. 

A similar formula as \eqref{Ds:def} using compactly supported kernels is used in \cite{TD16} to study some variational problems arising from peridynamics; however, for $s>1$, the operator \eqref{Ds:def} is usually not used directly in boundary value problems. Instead, for $s=n+\sigma>1$ with $n\in\N$ and $\sigma\in(0,1)$, one of the following options is preferred,
\begin{align} \label{notions}
 (i)\ \  (-\Delta)^{n}(-\Delta)^\sigma u(x),\qquad (ii)\ 
   (-\Delta)^\sigma(-\Delta)^{n} u(x),
\end{align}
\begin{align}\label{notions2}
(iii)\quad\left\lbrace\begin{aligned}
& (-\Delta)^\frac{n}{2} (-\Delta)^\sigma (-\Delta)^\frac{n}{2}u(x) && \text{ for $n$ even,} \\
& \sum_{i=1}^N (-\Delta)^\frac{n-1}{2}(\partial_i(-\Delta)^\sigma(\partial_i(-\Delta)^\frac{n-1}{2} u(x))) &&  \text{ for $n$ odd},
\end{aligned}\right.
\end{align}
see for example \cite{AJS16,AJS17,DG2016,RS15}.  These three possibilities are valid choices and their adequacy depends on the problem and the set of solutions that are being studied; for more details, see \cite[Remark A.4]{AJS17}. However, \eqref{notions} and \eqref{notions2} do not capture the full potential of the fractional Laplacian $(-\Delta)^s$.  Indeed, let $B_r$ denote the ball of radius $r$ centered at zero, $B:=B_1$, $\psi\in \cL^1_s\backslash \cL^1_{s-1}$ with $\psi=0$ in $B_{r}$ for some $r>1$, and consider the function 
\begin{align}\label{u:def}
 u(x):=  \frac{(-1)^n\Gamma(\frac{N}{2})}{\Gamma(\sigma)\Gamma(1-\sigma)\pi^{\frac{N}{2}}}\int_{\R^N\backslash\overline{B}}
 \frac{(1-|x|^2)_+^s\psi(y)}{(|y|^2-1)^s|x-y|^N}
\ dy+\chi_{\mathbb R^N\backslash\overline{B}}(x)\psi(x),\qquad x\in\R^N,
\end{align}
where $\chi_A$ is the characteristic function of a set $A\subset R^N$. Then, by \cite[Corollary 3.3]{AJS17}, $u\in\cC^\infty(B)\cap\cL^1_s$ is an $s$-harmonic function in the distributional sense, that is, 
\begin{align}\label{dist}
  \int_{\R^N}u(x)(-\Delta)^s\phi(x)\ dx=0 \qquad \text{ for all $\phi\in \cC^{\infty}_c(B)$,}
 \end{align}
where $\cC^{\infty}_c(B)$ denotes the space of smooth functions with compact support in $B$. In \eqref{dist}, $(-\Delta)^s$ can be understood pointwisely either with $(i)$, $(ii)$, or $(iii)$, since all these notions are equivalent for functions in $\cC^{\infty}_c(B)$ (see Proposition \ref{interchange} below). However, these pointwise evaluations cannot be applied to $u$ either because of its growth at infinity ($(-\Delta)^\sigma$ can only be applied to functions in $\cL_\sigma^1$) or because of its global regularity (the functions $(-\Delta)^\frac{n}{2} u$ and $(-\Delta)^n u$ may not exist in $\R^N\backslash B$, since $\psi$ is only required to be in $\cL_s^1$ and $u=\psi$ in $\R^N\backslash B$).
\medskip

The purpose of this paper is to show, with elementary calculations and in a self-contained manner, that \eqref{Ds:def} can be used to study boundary value problems. In particular, we show the equivalence between \eqref{Ds:def} and \eqref{notions}, \eqref{notions2} in suitable spaces, we provide an appropriate variational framework, and we prove that \eqref{u:def} is in fact a pointwise $s$-harmonic function using \eqref{Ds:def}.  Furthermore, we also include an asymptotic analysis of the operator $L_{m,s}$ as $s$ approaches $m$ from below, which we use to give an alternative (more elementary) proof of the fact that $L_{m,n}$ reduces to the polyharmonic operator for $n\in\N$.  For completeness, we also provide an elementary proof that the Fourier symbol of \eqref{Ds:def} is in fact $|\xi|^{2s}$, justifying the precise value of the normalizing constant \eqref{cNms:def}.  Our proofs are mainly based on pointwise calculations and known identities for Laplace transforms, Gamma functions, and combinatorial coefficients. 

\medskip

To present our results in a unified manner, we introduce some notation. For $n\in \N_0$, $\sigma\in(0,1]$, $s=n+\sigma$, and $U\subset \R^N$ open, we write $\cC^n(U)$ to denote the space of $n$-times continuously differentiable functions in $U$ and $C^s(U)$ to denote the space of functions in $\cC^n(U)$ whose derivatives of order $n$ are {locally} $\sigma$-H\"older continuous (or {locally} Lipschitz continuous if $\sigma=1$) in $U$.  Observe that $C^1(U)$ denotes the space of {locally} Lipschitz continuous functions, which is different from $\cC^1(U)$. We also use the norm
\begin{align*}
 \|u\|_{C^{2s+\beta}(U)}:=\|u\|_{\cC^n(U)}+\sup_{x,y\in U}\frac{|u(x)-u(y)|}{|x-y|^{\sigma}},
\end{align*}
where $\|u\|_{\cC^n(U)}$ is the usual supremum norm associated to $\cC^n(\overline{U})$.

\medskip

Our first result states that $L_{m,s}:C^{2s+\beta}(U)\cap \cL^1_s\to C^{\beta}(U)$. For sets $A,B\subset \R^N$, denote $A\subset\subset B$, if $\overline{A}$ is compact and contained in $B$.
\begin{lemma}\label{welldef:lemma}
Let $U\subset \R^N$ open, $V\subset\subset U$ open, $m\in\N$, $\beta\in(0,1)$, $s\in(0,m)$, and $u\in C^{2s+\beta}(U)\cap \cL^1_s$, then
\begin{align}\label{Holder:n}
\|L_{m,s} u\|_{C^{\beta}(V)}\leq C\|u\|_{C^{2s+\beta}(V)}\qquad \text{for some }C(N,m,s)=C>0.
\end{align}
In particular, $\|L_{m,s} u\|_{C^{\beta}(U)}\leq C\|u\|_{C^{2s+\beta}(U)}$, whenever $\|u\|_{C^{2s+\beta}(U)}$ is finite.
\end{lemma}
The proof is based on a higher-order extension of \cite[Proposition 2.5]{S07} using a multivariate Taylor expansion and combinatorial identities.

Our next result relates the operator $L_{m,s}$ with the pointwise notion $(ii)$ given above. Here $W^{k,1}_{loc}(\R^N)$ is the usual Sobolev space of $k$-weakly differentiable locally integrable functions.
\begin{thm}\label{new:main:thm}
Let $m,n\in \N_0$, $n<m$, $\sigma\in(0,1)$, $s=n+\sigma$, $U\subset\R^N$ be an open set, and $u\in C^{2s+\beta}(U)\cap W^{2n,1}_{loc}(\R^N)$ such that $(-\Delta)^n u\in \cL^1_\sigma$, $(-\Delta)^iu\in\cL^1_{s-i-\frac{1}{2}}$, and $|\nabla (-\Delta)^i u|\in\cL^1_{s-i-1}$ for $i\in\{0,\ldots n-1\}$. Then 
 \begin{align}\label{new:claim}
L_{m,s} u = (-\Delta)^\sigma (-\Delta)^n u\qquad \text{ in }U. 
 \end{align}
 In particular, \eqref{new:claim} holds in $U=\R^N$ for all $u\in \cC^{\infty}_c(\R^N)$. 
\end{thm}
We remark that, for functions in $\cC^\infty_c(\R^N)$, one can freely interchange derivatives and the fractional Laplacian $(-\Delta)^\sigma$, see Proposition \ref{interchange} below, therefore Theorem \ref{new:main:thm} implies that $L_{m,s}u$ is equivalent to \eqref{notions} and \eqref{notions2} for  $u\in \cC^{\infty}_c(\R^N)$.  We also note that the assumptions of Theorem \ref{new:main:thm} are satisfied by the \emph{fundamental solution} $F_{N,s}$ of $(-\Delta)^s$ in $\R^N$, see \cite[Chapter 5 Lemma 25.2]{SKM93} or \cite[Section 5]{AJS16} for the exact formula of $F_{N,s}$.  Furthermore, observe that $m\in\N$ can be arbitrarily large in \eqref{new:claim} and only the restriction $s<m$ is relevant.  In fact, we have the following result.
\begin{lemma}\label{mton}
 Let $U\subset \R^N$, $s>0$, $\beta>0$, and $u\in C^{2s+\beta}(U)\cap \cL^1_{s}$. Then $L_{m,s} u = L_{n,s} u$ in $U$ for all $n,m\in\N$ such that $m>n>s$.
\end{lemma}

As a consequence of Theorem \ref{new:main:thm} and Lemma \ref{mton}, we have the following \emph{pointwise} equivalences. Let $\cH^s_0(U):=\{u\in H^s(\R^N)\::\: u=0\text{ in }\R^N\backslash U\}$, where $H^s(\R^N)$ is the usual Sobolev space for $s>0$.

\begin{cor}\label{e:d:cor}
 Let $n,m\in \N_0$, $n<m$, $\beta,\sigma\in(0,1)$, $s=n+\sigma$, $U\subset \R^N$ open bounded Lipschitz domain, and $u\in C^{2s+\beta}(U)$.
 \begin{enumerate}
  \item[(a)] If $u\in \cL^1_{\sigma}$, then $L_{m,s} u=(-\Delta)^n(-\Delta)^\sigma u$ in $U$.
  \item[(b)] If $n$ is even and $u\in \cH^s_0(U)$, then $L_{m,s} u=(-\Delta)^\frac{n}{2}(-\Delta)^\sigma(-\Delta)^\frac{n}{2} u$ in $U$.
  \item[(c)] If $n$ is odd and $u\in \cH^s_0(U)$, then 
   $L_{m,s} u=\sum_{i=1}^N (-\Delta)^\frac{n-1}{2}(\partial_i(-\Delta)^\sigma(\partial_i(-\Delta)^\frac{n-1}{2} u))$ in $U$.
   \end{enumerate}
\end{cor}
 The proof relies on the fundamental theorem of calculus of variations, Lemma \ref{ibyp:ours}, and the following analogue of integration by parts.
\begin{lemma}\label{ibyp}
 Let $m\in \N$, $\beta\in(0,1)$, $s\in(0,m)$, $U\subset\R^N$ open, and $u\in C^{2s+\beta}(U)\cap \cL^1_s$. Then 
\begin{align*}
 \int_{\R^N} u(x)L_{m,s} \varphi(x)\ dx=\int_{\R^N} L_{m,s}u(x) \varphi(x)\ dx\qquad \text{ for all }\varphi\in \cC^\infty_c(U).
\end{align*}
\end{lemma}
We also show that, if $s\in\N$, then $L_{m,s}$ is the usual polyharmonic operator. 
\begin{thm}\label{van:cor}
Let $m,n\in\N$ such that $n<m$, $U\subset \R^N$ open, and $u\in \cC^{2n}(U)\cap \cL^1_{n}(\R^N)$, then $L_{m,n}u =(-\Delta)^n u= (-\sum_{i=1}^N\partial_{ii})^n u$ in $U$. 
Moreover, for $\eta,\beta\in(0,1)$ and $x\in U$,
\begin{equation}\label{asymp}
\begin{aligned}
\lim_{s\to0^+}L_{m,s}u(x)&=u(x) \quad\quad\quad\quad\text{ for all $u\in C^{\beta}(U)\cap L^{\infty}(\R^N)$,}\\
\lim_{s\to m^-}L_{m,s}u(x)&=(-\Delta)^mu(x) \quad\text{ for all $u\in \cC^{2m}(U)\cap \cL^1_{m-\eta}$.}
\end{aligned}
\end{equation}
\end{thm}
This result shows the consistency of the exact values of $c_{N,m,s}$. For a similar asymptotic study in the case $s\in(0,1)$, we refer to \cite[Proposition 4.4]{NPV11}. Observe also that $L_{m,n}$ is well defined in $\cC^{2n}(U)\cap \cL^1_{n}(\R^N)$, which is larger than $C^{2n+\beta}(U)\cap \cL^1_{n}(\R^N)$.  We remark that the representation of local operators as hypersingular integrals holds in much more generality, see \cite[Chapter 5, Section 26.6]{SKM93}, where different techniques from ours are used.

\medskip

We are ready to show that \eqref{u:def} gives rise to a \emph{pointwise} $s$-harmonic function with prescribed (nonlocal) boundary values. Here $H^s(U):=\{u\chi_U\::\: u\in H^s(\R^N)\}$.
\begin{thm}
 Let $s\in(0,\infty)\backslash \N$, $\beta\in(0,1)$, $\psi\in \cL^1_s$ with $\psi=0$ in $B_{r}(0)$ for some $r>1$, and let $u$ be given by \eqref{u:def}. Then $u\in C^\infty(B)\cap C^s(B_r(0))\cap H^s(B_{\rho}(0))\cap \cL^1_s$, $\rho\in(1,r)$ is the unique pointwise solution in $C^{2s+\beta}(B)\cap C^s(\overline{B})\cap H^s(B)$ of 
 \begin{align*}
(-\Delta)^s u(x)=0\qquad\text{ for all }x\in B\qquad \text{ and }\qquad u=\psi\quad \text{ in }\R^N\backslash B.  
 \end{align*}
 Here $(-\Delta)^s u:= L_{m,s} u$ for any $m\in\N$ with $m>s$.
\end{thm}
The proof follows immediately from \cite[Corollary 3.3]{AJS17} (recall \eqref{dist} above), Theorem \ref{new:main:thm}, and Lemma~\ref{ibyp}.

We remark that an alternative definition for $L_{m,s}$ using the principal value integral can be obtained as follows. For $m\in \N$ and $u:\R^N\to \R$ let 
\begin{equation*}
\delta^{+}_mu(x,y):=\frac{1}{2}\binom{2m}{m}u(x)+\sum_{k=1}^\m (-1)^k { \binom{2m}{m-k}} u(x+ ky), \qquad x,y\in \R^N.
\end{equation*}
 Let $U\subset \R^N$ open, $s\in(0,m)$, and $\beta\in(0,1)$. For $u\in C^{2s+\beta}(U)\cap \cL^1_s$ and $x\in U$, we have that
 \begin{equation}\label{def-frac}
 L_{m,s} u(x):=c_{N,m,s}p.v.\int_{\R^N}\frac{\delta^{+}_{m}u(x,y)}{|y|^{N+2s}}\ dy=c_{N,m,s}\lim_{\epsilon\to0^+}\int_{B_{1/\epsilon}(0)\setminus B_{\epsilon}(0)}\frac{\delta^{+}_{m}u(x,y)}{|y|^{N+2s}}\ dy,
 \end{equation}
 with $c_{N,m,s}$ as in \eqref{cNms:def}. Observe that \eqref{def-frac} and \eqref{Ds:def} are equivalent in $\cL^1_s\cap C^{2s+\beta}(U)$ with $U\subset \R^N$ open, via a change of variables.

\medskip

To connect $L_{m,s}$ to an appropriate variational framework, we next study an equivalent \emph{scalar product} for $H^s(\R^N)$, $s>0$, using the difference operator $\delta_m$.  To state this result, we recall first the scalar product associated to $(-\Delta)^s$ as introduced in \cite{AJS16}. For $n\in\N$, $\sigma\in(0,1)$, $s=n+\sigma$, and $u,v\in H^{s}(\R^N)$, let 
\begin{align}
 \cE_{\sigma}(u,v)&:=
\frac{c_{N,1,\sigma}}{2}\int_{\R^N}\int_{\R^N}\frac{(u(x)-u(y))(v(x)-v(y))}{|x-y|^{N+2\sigma}}\ dx\ dy,\nonumber\\
\cE_{s}(u,v)&:=\left\{\begin{aligned}
&\cE_{\sigma}((-\Delta)^{\frac{n}{2}}  u,(-\Delta)^{\frac{n}{2}} v),&& \quad \text{if $n$ is even,}\\
&\sum_{k=1}^{N}\cE_{\sigma}(\partial_k (-\Delta)^{\frac{n-1}{2}} u,\partial_k (-\Delta)^{\frac{n-1}{2}} v),&&
\quad  \text{if $n$ is odd,}
\end{aligned}\right.\label{bilin:def}
\end{align}

In particular (see \cite[Proposition 3.1]{AJS16}, here $\cF$ stands for the Fourier transform),
\begin{align}\label{fourier}
\cE_{s}(u,v)=\int_{\R^N} |\xi|^{2s}\cF u(\xi)\cF v(\xi)\ d\xi\qquad \text{ for all }s>0.
\end{align}

If $U\subset \R^N$ is a bounded open Lipschitz set, the space $\cH^{s}_0(U)$ equipped with the norm
\begin{align*}
\|u\|_{\cH^s_0(U)}:=(\sum_{|\alpha|\leq m}\|\partial^{\alpha} u\|_{L^2(U)}^2+\cE_{s}(u,u))^{\frac{1}{2}} 
\end{align*}
is a Hilbert space and, using \eqref{bilin:def}, a consistent notion of weak solution can be defined; see \cite{AJS16}, where existence, regularity, and positivity of weak solutions to boundary value problems is studied. 

Then, an equivalent scalar product using the difference operators $\delta_m$ can be defined as follows. For $u,v\in H^{s}(\R^N)$ let $m\in\N$ such that $s\in(0,2m)$ and let
\begin{equation*}
\cE_{2m,s}(u,v):=\frac{c_{N,2m,s}}{2}\int_{\R^N}\int_{\R^N}\frac{\delta_{m}u(x,y) \delta_{m}v(x,y)}{|y|^{N+2s}}\ dxdy.
\end{equation*}
We have the following result.
\begin{thm}\label{bilinear}
Let $s>0$, $n,m\in \N$, and $s<n\leq 2m$. Then $\cE_{2m,s}(u,v)=\cE_s(u,v)$ for $u,v\in H^s(\R^N)$ and $\int_{\R^N}L_{n,s}u(x)\ v(x)\ dx=\cE_{2m,s}(u,v)$ for $u,v\in C^\infty_c(\R^N)$.
\end{thm}
In virtue of Corollary \ref{e:d:cor} and Theorem \ref{bilinear}, the results from \cite{AJS16,AJS17} extend trivially to $L_{m,s}$ and $\cE_{2m,s}$, note however that $L_{m,s}$ can be applied to a larger set of functions than \eqref{notions} and \eqref{notions2}. As a consequence, \cite[Theorems 1.1, 1.4, 1.5]{AJS17} can be generalized to allow outside data in $\R^N\backslash B$ which belongs to $\cL^1_s$ instead of $\cL^1_\sigma$ using the pointwise evaluation $(-\Delta)^s:=L_{m,s}$ with $m>s$.

\medskip

For our last result, we \emph{directly} show in detail that the Fourier symbol of $L_{m,s}$ is $|\xi|^{2s}$. This fully justifies the precise values of the normalizing constant \eqref{cNms:def}, which plays an essential role in our proofs.
\begin{thm}\label{main:thm}
Let $m\in\N$, $s\in(0,m)$, and $L_{m,s}:\cC_c^\infty(\R^N)\to L^2(\R^N)$ be given by \eqref{Ds:def}. Then $L_{m,s} u=\cF^{-1}(|\cdot|^{2s}\cF(u))$ in $\R^N$ for any $u\in \cC_c^\infty(\R^N)$.
\end{thm}
We note that this statement is known and we include a different elementary proof for completeness. For an alternative proof, see \cite[Chapter 5, Lemma 25.3 and Theorem 26.1]{SKM93}, which uses Fourier series and analytic continuation.

\medskip

To conclude this introduction, let us suggest an heuristic interpretation for $L_{m,s}$. Intuitively, \eqref{Ds:def} implies that any power $s>0$ of the Laplacian operator $\Ds$ can be seen as another power $\theta\in(0,1)$ of a polylaplacian ${(-\Delta)}^m$ in such a way that
$s=\theta m$. Indeed, with this notation, one can rewrite \eqref{Ds:def} as 
\begin{align*}
L_{m,s} u(x)=\frac{c_{N,\m,s}}{2}\int_{\R^N} \frac{\delta_\m u(x,y)}{|y|^{2\theta m}} \ \frac{dy}{|y|^N},\qquad x\in \R^N,
\end{align*}
which is a nonlocal average in $\R^N$ of the $2m$-th order difference quotient with an altered exponent, as it happens for the standard fractional Laplacian \eqref{frac-lapl}.
In view of Theorem \ref{van:cor}, we have the convergence to the polylaplacian ${(-\Delta)}^m$ as $\theta\to 1^-$. Also, the Fourier symbol of the operator can be seen as $|\xi|^s=(|\xi|^m)^{\theta}$.
In this spirit, Lemma \ref{mton} implies the equivalence between considering a power $\theta$ of the polylaplacian of order $2m$ or a power $\tau\in(0,1)$ of the polylaplacian of order $2n$, as long as $\theta m=s=\tau n$.

\medskip

The paper is organized as follows.  We collect first some preliminary results in Section \ref{Preliminaries}, where in particular the proof of Lemma \ref{welldef:lemma} can be found.  Section \ref{Equiv:sec} contains the proofs of all the other results stated in the introduction, except for Theorem \ref{main:thm}, to which Section \ref{n:c:sec} is devoted.

\section{Preliminaries}\label{Preliminaries}

\subsection{The difference operator}

The next Lemma follows the ideas from \cite[Lemma 1]{TD16}.
\begin{lemma}\label{iteration}
Let $\m\in \N_0$ and $u:\R^N\to \R$, then 
\begin{equation}\label{iteration1}
\delta_{\m+1}(x,y)= \delta_{\m} [\delta_1 u(\cdot,y)](x,y)\qquad \text{  for all }x,y\in \R^N.
\end{equation}
In particular, if $f(t):=\exp(it)$ for $t\in \R$, then
\begin{equation}\label{iteration2}
\delta_\m f(0,t)=2^\m(1-\cos(t))^\m \qquad \text{ for $t\in \R$}.
\end{equation}
\end{lemma}
\begin{proof}
By definition,
\begin{align*}
&\delta_{\m}[\delta_1 u(\cdot,y)](x,y)=\sum_{k=-\m}^{\m} (-1)^k{ \binom{2m}{m-k}}[-u(x+(k-1)y)+2u(x+ky)-u(x+(k+1)y)]\\
&=\sum_{k=-\m}^{\m} (-1)^{k-1}{ \binom{2\m}{\m-1-(k-1)}}u(x+(k-1)y)+2\sum_{k=-\m}^{\m} (-1)^{k}{ \binom{2m}{m-k}}u(x+ky)\\
&\qquad  +\sum_{k=-\m}^{\m} (-1)^{k+1}{ \binom{2\m}{\m+1-(k+1)}}u(x+(k+1)y)\\
&=\sum_{k=-\m-1}^{\m-1} (-1)^{k}{ \binom{2\m}{\m-1-k}}u(x+ky)+2\sum_{k=-\m}^{\m} (-1)^{k}{ \binom{2m}{m-k}}u(x+ky)\\
&\qquad+\sum_{k=-\m+1}^{\m+1} (-1)^{k}{ \binom{2\m}{\m+1-k}}u(x+ky)\\
&= (-1)^{-\m-1}u(x-(\m+1)y)+ (-1)^{-\m}2\m u(x-\m y) + 2 (-1)^{-\m} u(x-\m y)\\
&\qquad+\sum_{k=-\m+1}^{\m-1} (-1)^{k}\Big[{ \binom{2\m}{\m-1-k}}+2{ \binom{2m}{m-k}}+{ \binom{2\m}{\m+1-k}}\Big]u(x+ky)\\
&\qquad+2(-1)^{\m}u(x+\m y) + (-1)^{\m}2\m u(x+\m y)+  (-1)^{\m+1}u(x+(\m+1)y)\\
&= (-1)^{-\m-1}u(x-(\m+1)y)+ (-1)^{-\m}2(\m+1) u(x-\m y)\\
&+\sum_{k=-\m+1}^{\m-1} (-1)^{k}{\binom{2\m+2}{\m+1-k}}u(x+ky)+ (-1)^{\m}(2\m+1) u(x+\m y)+  (-1)^{\m+1}u(x+(\m+1)y)\\
&= \sum_{k=-\m-1}^{\m+1} (-1)^k{ \binom{2(m+1)}{m+1-k}}u(x+ky)= \delta_{\m+1} u(x,y),
\end{align*}
where we used that ${ \binom{2\m}{\m-1-k}}+2{ \binom{2m}{m-k}}+{ \binom{2\m}{\m+1-k}} = {\binom{2\m+2}{\m+1-k}},$ and \eqref{iteration1} follows. We now argue \eqref{iteration2} by induction on $m$. For $\m=1$ it follows that 
\begin{align}\label{df1}
\delta_1f(x,t)=-f(x-t)+2f(x)-f(x-t)=-e^{-it}f(x)+2f(x)-e^{it}f(x)
\end{align}
and, in particular, $\delta_1f(0,t)=-e^{-it}+2-e^{it}=2(1-\cos(t))$, since $e^{-it}+e^{it}=2\cos(t)$.  Now, assume that \eqref{iteration2} holds for some $\m\in \N$, then, by \eqref{df1}, \eqref{iteration1},
\begin{align*}
\delta_{\m+1}f(0,t)&= \delta_{\m}[\delta_1f(\cdot,t)](0,t)=(-e^{-it}+2-e^{it})\delta_{\m}f(0,t)\\
&= 2(1-\cos(t))2^{\m}(1-\cos(t))^{\m}=2^{\m+1}(1-\cos(t))^{\m+1},
\end{align*}
and the claim follows. 
\end{proof}

\begin{lemma}\label{calc:lemma}
	Let $m,n\in\N_0$ with $n<m$, then
	\begin{align}\label{calc}
	\sum_{k=-\m}^\m (-1)^k { \binom{2m}{m-k}} k^{2n}=0\qquad\text{ and }\qquad \sum_{k=-\m}^\m (-1)^k { \binom{2m}{m-k}} k^{2m}=(-1)^\m (2m)!
	\end{align}
\end{lemma}
\begin{proof}
	First we claim that, for $n\in \N$, $g\in \cC^{2n}(\R)$, and $x,t\in \R$,
	\begin{equation}\label{calc:lemma:1}
	\delta_{n}g(x,t)=(-1)^nt^{2n}\int_0^1\ldots\int_0^1 g^{(2n)}(x+\sum_{k=1}^{n}((t_{k,1}-t_{k,2})t)\ dt_{1,1}\ldots dt_{n,1}dt_{1,2}\ldots dt_{n,2}.
	\end{equation}
	We argue by induction on $n$. For $n=1$ the claim follows, since 
	\begin{align*}
	\delta_1g(x,t)=t\int_{0}^{1}g'(x-t+t_1t)\ dt_1 -t\int_0^1g'(x+t_1t)\ dt_1=-t^2\int_0^1\int_0^1g''(x+t_{1,1}t-t_{1,2}t)\ dt_{1,2}\ dt_{1,1}.
 	\end{align*}
Next, if $n\in \N$ is such that \eqref{calc:lemma:1} holds, then, by Lemma \ref{iteration},
	\begin{align*}
	\delta_{n+1}&g(x,t)=\delta_n[\delta_1(\cdot,t)](x,t)=-t^{2}\int_0^1\int_0^1\delta_n[g''(\cdot+(t_{1,1}-t_{1,2})t)](x,t)\ dt_{1,1}dt_{1,2} \\
	&=(-1)^{n+1}t^{2n+2}\int_0^1\ldots\int_0^1g^{(2n+2)}(\cdot+\sum_{k=2}^{n+1}((t_{k,1}-t_{k,2})t+(t_{1,1}-t_{1,2})t)](x,t)\ dt_{1,1}\ldots dt_{n+1,2}.
	\end{align*}
	Therefore \eqref{calc:lemma:1} holds for all $n\in\N$.  In particular, by continuity, 
	\begin{equation}\label{calc:lemma:2}
	\lim_{t\to 0}\frac{\delta_{n}g(0,t)}{t^{2n}}=(-1)^n g^{(2n)}(0)\int_0^1\ldots\int_0^1\ dt_{1,1}\ldots dt_{n,2}=(-1)^n g^{(2n)}(0).
	\end{equation}
	Let $m,n\in \N$ with $n<m$ and $g_s(t):=t^{2s}$ for $t\in\R$ and $s>0$, then, by \eqref{calc:lemma:2},
	\begin{align*}
 	\sum_{k=-\m}^\m (-1)^k { \binom{2m}{m-k}} k^{2n}&=\delta_m g_n(0,1)=\lim_{t\to 0} \frac{\delta_m g_n(0,t)}{t^{2n}}=(-1)^n g_n^{(2n)}(0)=0,\\
 	\sum_{k=-\m}^\m (-1)^k { \binom{2m}{m-k}} k^{2m}&=\delta_\m g_m(0,1)=\lim_{t\to 0}\frac{\delta_m g_m(0,t)}{t^{2m}}=(-1)^m g_m^{(2\m)}(0)=(-1)^m (2m)!,
 	\end{align*}
 	as claimed.
\end{proof}

For our next result, recall that for $x\in\R^N$, $\eta>0$, $j\in\N$, $v\in \cC^{2j}(B_\eta(x))$, and $h\in B_\eta(0)$, the \emph{multivariate Taylor expansion} yields that
\begin{align}\label{mte}
 v(x+h)&=\sum_{|\alpha|\leq 2j-1}\frac{\partial^\alpha v(x)}{\alpha!}h^\alpha + \sum_{|\alpha|=2j}\frac{\partial^\alpha v(x+\theta h)}{\alpha!}h^\alpha
\end{align}
for some $\theta(v,x,h,j)=\theta\in(0,1),$ where $\alpha=(\alpha_1,\ldots, \alpha_N)\in \N^N$, $\alpha!=\alpha_1!\ldots\alpha_N!$, 
\begin{align*}
 \partial^\alpha u = \frac{\partial^{|\alpha|} u}{\partial x_1^{\alpha_1}\cdots\partial x_N^{\alpha_N}},\qquad  y^\alpha = \prod_{i=1}^N y_i^{\alpha_i},
\end{align*}
and $B_\eta(x)$ denotes the open ball centred at $x$ of radius $\eta$.

\begin{lemma}\label{mte:lemma}
Let $x\in\R^N$, $r>0$, $u\in \cC^{2m}(B_r(x))$, $m,j\in \N$, $m\geq j$, and $y\in B_r(0)$.  There is $\theta(u,x,y,j)=\theta\in[0,1]$ such that 
\begin{align}\label{mte:eq}
\delta_m u(x,y)
=\sum_{|\alpha|=2{j}}\frac{1}{\alpha!} \sum_{k=-\m}^\m (-1)^k { \binom{2m}{m-k}}k^{2m} \partial^\alpha u(x+k(\theta y))y^\alpha.
\end{align}
 \end{lemma}
\begin{proof}
Let $u\in \cC^{2{j}}(B_r(x))$, $n\in \N$, $n< 2{j}$, $\alpha\in \N^N$, $|\alpha|=n$. Then,
\begin{align*}
 \partial_y^\alpha \delta_m u(x,0)
 &= \sum_{k=-\m}^\m (-1)^k { \binom{2m}{m-k}}k^{n} [\partial^\alpha u](x+ky)|_{y=0}=\partial^\alpha u(x)\sum_{k=-\m}^\m (-1)^k { \binom{2m}{m-k}}k^{n}=0,
 \end{align*}
 by Lemma \ref{calc:lemma}.  Therefore, applying \eqref{mte} to $y\mapsto\delta_m u(x,y)$ at $y=0$, we have that, for some $\theta\in[0,1]$ and for all $h\in B_r(0)$, that
\begin{align*}
 \delta_m u(x,h)&=\sum_{|\alpha|=2{j}}\frac{\partial^\alpha \delta_m u(x,\theta h)}{\alpha!}h^\alpha=\sum_{|\alpha|=2{j}}\frac{1}{\alpha!} \sum_{k=-\m}^\m (-1)^k { \binom{2m}{m-k}}k^{2m} \partial^\alpha u(x+k(\theta h))h^\alpha,
\end{align*}
and \eqref{mte:eq} follows. 
\end{proof}

\begin{lemma}\label{mte:lemma:2}
Let $U\subset \R^N$, $x\in U$, $\varepsilon>0$, and $u\in \cC^{2m}(U)$. There is $\rho(U,x,m,u,\varepsilon)=\rho>0$ such that
 \begin{align}\label{mte:eq:2}
  \delta_m u(x,y)
  =R(x,y)+\sum_{|\alpha|=2m}\frac{(-1)^m(2m)!}{\alpha!} \partial^\alpha u(x)y^\alpha \qquad \text{ for all }y\in B_\rho(0),
 \end{align}
 where $|R(x,y)|\leq C \varepsilon |y|^{2m}$ for some $C(N,m)=C>0$.
 \end{lemma}
\begin{proof}
 Let $x\in U$ and $\rho_0\in(0,1)$ such that $x+ky\in U$ for all $k\in\{-m\ldots,m\}$ and $y\in B_{\rho_0}(0)$.  By Lemmas \ref{mte:lemma} and \ref{calc:lemma}, there is $\theta(x,y,u)=\theta\in[0,1]$ such that \eqref{mte:eq:2} holds for $y\in B_{\rho_0}(0)$ with
 \begin{align*}
  R(x,y)=\sum_{|\alpha|=2m}\frac{1}{\alpha!} \sum_{k=-\m}^\m (-1)^k { \binom{2m}{m-k}}k^{2m} (\partial^\alpha u(x+k\theta y)-\partial^\alpha u(x))y^\alpha.
 \end{align*}
 Then, since $u\in \cC^{2m}(U)$, there is $\rho(U,x,m,u,\varepsilon)=\rho\in(0,\rho_0]$ such that
 \begin{align*}
  |R(x,y)|\leq \sum_{|\alpha|=2m}\frac{1}{\alpha!} \sum_{k=-\m}^\m { \binom{2m}{m-k}}k^{2m} \varepsilon |y^\alpha|\leq C \varepsilon|y|^{2m}\qquad \text{ for all }y\in B_{\rho}(0)
 \end{align*}
for some $C>0$ depending only on $N$ and $m$.
\end{proof}

We continue with the proof of Lemma \ref{welldef:lemma}, where we extend the arguments in \cite[Proposition 2.5]{S07}.

\begin{proof}[Proof of Lemma \ref{welldef:lemma}]

Let $U\subset \R^N$ open, $n,m\in \N$ such that $s\in(n-1,n)$ with $n\leq m$, $\sigma:=s-n+1$, fix $V\subset\subset U$ open, let $x,z\in V$, $x\neq z$, and
\begin{align*}
0 < r < \min\Big\{ |x-z| , \frac{\dist(x,\partial V)}{2m} ,  \frac{\dist(z,\partial V)}{2m}  \Big\}.
\end{align*}
In particular, for all $k\in\{1,\ldots,m\}$, $\theta\in[0,1]$, and $y\in B_r(0)$, we have that $x+ky\theta,z+ky\theta\in V$. In the following we use $C>0$ to denote possibly different constants depending at most on $N$, $m$, and $s$.

Assume first that $2\sigma+\beta\in(0,1)$, $u\in C^{2s+\beta}(U)\cap \cL^1_s$, and $y\in B_r(0)$. By Lemma \ref{mte:lemma}, there is $\theta_{x,y}\in[0,1]$ such that
 \begin{align}\label{dec1}
\delta_m u(x,y)=\sum_{|\alpha|=2n-2}\frac{1}{\alpha!} \sum_{k=-\m}^{\m} (-1)^k { \binom{2m}{m-k}}k^{2n-2} \partial^\alpha u(x+k y\theta_{x,y})y^\alpha;
\end{align}
moreover,
\begin{align}
 &|L_{m,s}u(x)-L_{m,s}u(z)|
 \leq \frac{c_{N,\m,s}}{2}\int_{\R^N} |\delta_\m u(x,y)-\delta_\m u(z,y)| |y|^{-N-2s}\ dy.\label{eq1}
\end{align}
Note that, for all $k\in\{1,\ldots,m\}$, $\theta\in[0,1]$, and $y\in B_r(0)$, we have that 
\begin{align}\label{H:est}
|x-z+ky(\theta_{x,y}-\theta_{z,y})|\leq |x-z|+2kr<(2m+1)|x-z|.
\end{align}
 Therefore, by \eqref{dec1}, \eqref{eq1}, \eqref{H:est}, and the fact that $u\in C^{2s+\beta}(U)$,
\begin{align}
\int_{\R^N\backslash B_r(0)} &\frac{|\delta_\m u(x,y)-\delta_\m u(z,y)|}{|y|^{N+2s}}\ dy\leq C\|u\|_{C^{2s+\beta}(V)} \int_{\R^N\backslash B_r(0)} \sum_{|\alpha|=2n-2} \frac{|x-z|^{2\sigma+\beta}}{|y|^{N-2(-s+n-1)}}\ dy\nonumber\\
&\hspace{2cm}=\|u\|_{C^{2s+\beta}(V)}C |x-z|^{2\sigma+\beta}r^{2(-s+n -1)}{\leq }C\|u\|_{C^{2s+\beta}(V)}|x-z|^{\beta}.\label{est2}
\end{align}
Furthermore, note that \eqref{calc} implies
\begin{align}\label{bcid}
\sum_{k=1}^{m}{\binom{2m}{m-k}}(-1)^{k+1}=\frac{1}{2}{\binom{2m}{m}}.
\end{align}
Therefore, by \eqref{bcid} and \eqref{dec1} for $n=1$ and by \eqref{calc} and \eqref{dec1} for $n>1$,
\begin{align*}
|\delta_m u(x,y)|&= |\sum_{|\alpha|=2n-2}\frac{1}{\alpha!} \sum_{k=-\m}^{\m} (-1)^k { \binom{2m}{m-k}}k^{2n-2} \partial^\alpha u(x+k\theta y)y^\alpha|\\
&=|\sum_{|\alpha|=2n-2}\frac{1}{\alpha!}
\Big(\sum_{k=1}^{\m} (-1)^k { \binom{2m}{m-k}}k^{2n-2} (\partial^\alpha u(x)-\partial^\alpha u(x+k\theta y))\\
&\qquad +\sum_{k=1}^{\m} (-1)^k { \binom{2m}{m-k}}k^{2n-2} (\partial^\alpha u(x)-\partial^\alpha u(x-k\theta y))\Big)
y^\alpha|\\
&\leq C\|u\|_{C^{2s+\beta}(V)}\sum_{|\alpha|=2n-2}\frac{2}{\alpha!} \sum_{k=1}^{\m}{ \binom{2m}{m-k}}k^{2n-2+2\sigma+\beta} |y|^{2\sigma+\beta+2n-2}\\
&= C\|u\|_{C^{2s+\beta}(V)} |y|^{2s+\beta}.
\end{align*}
Thus,
\begin{equation}\label{est1}
\begin{aligned}
\int_{B_r(0)} \frac{|\delta_\m u(x,y)-\delta_\m u(z,y)|}{|y|^{N+2s}}\ dy
&\leq C\|u\|_{C^{2s+\beta}(V)} \int_{B_r(0)} |y|^{\beta-N}\ dy\\
&=C\|u\|_{C^{2s+\beta}(V)} r^{\beta} {\leq} C\|u\|_{C^{2s+\beta}(V)} |x-z|^{\beta}.
 \end{aligned}
 \end{equation}
Then \eqref{Holder:n} follows from adding \eqref{est2} and \eqref{est1}, since $x,z\in V$ were arbitrarily chosen.

The proof for $2\sigma+\beta\geq 1$ is similar. In this case, $2s+\beta\geq 2n-1$ and, by \eqref{mte}, there is $\theta_{x,y}\in[0,1]$ such that
 \begin{align*}
\delta_m u(x,y)=\sum_{|\alpha|=2n-1}\frac{1}{\alpha!} \sum_{k=-\m}^{\m} (-1)^k { \binom{2m}{m-k}}k^{2n-1} \partial^\alpha u(x+k y\theta_{x,y})y^\alpha.
\end{align*}
Let $\gamma:=2s+\beta-2n+1\in(0,1)$, then, arguing similarly as before, we obtain that
\begin{align}
\int_{\R^N\backslash B_r(0)}&\frac{|\delta_\m u(x,y)-\delta_\m u(z,y)|}{|y|^{N+2s}}\ dy\leq C\|u\|_{C^{2s+\beta}(V)} \int_{\R^N\backslash B_r(0)} \sum_{|\alpha|=2n-1}\frac{1}{\alpha!}  \frac{|x-z|^{\gamma}}{|y|^{N+2s-2m+1}}\ dy\nonumber\\
&\qquad \qquad=\|u\|_{C^{2s+\beta}(V)}C |x-z|^{\gamma}r^{2(m-s)-1}{\leq }C\|u\|_{C^{2s+\beta}(V)}|x-z|^{\beta}.\label{est22}
\end{align}

On the other hand, 
\begin{align*}
|\delta_m u(x,y)|&= \Big|\sum_{|\alpha|=2n-1}\frac{1}{\alpha!} \sum_{k=-\m}^{\m} (-1)^k { \binom{2m}{m-k}}k^{2n-2} \partial^\alpha u(x+k\theta y)y^\alpha\Big|\\
&=\Big|\sum_{|\alpha|=2n-1}\Big(\ \frac{1}{\alpha!}
(\sum_{k=1}^{\m} (-1)^k { \binom{2m}{m-k}}k^{2n-1} (\partial^\alpha u(x)-\partial^\alpha u(x+k\theta y))\\
&\qquad +\sum_{k=1}^{\m} (-1)^k { \binom{2m}{m-k}}k^{2n-1} (\partial^\alpha u(x)-\partial^\alpha u(x-k\theta y))
y^\alpha\Big)\Big|\\
&\leq C\|u\|_{C^{2s+\beta}(V)}\sum_{|\alpha|=2n-1}\frac{2}{\alpha!} \sum_{k=1}^{\m}{ \binom{2m}{m-k}}k^{2n-1+\gamma} |y|^{\gamma+2n-1}= C\|u\|_{C^{2s+\beta}(V)} |y|^{2s+\beta},
\end{align*}
Then \eqref{est1} holds also in this case and, together with \eqref{est22}, this implies \eqref{Holder:n}. Finally, since $V\subset\subset U$ is an arbitrary open subset, \eqref{Holder:n} also holds for $V=U$, if $\|u\|_{C^{2s+\beta}(U)}$ is finite.
\end{proof}

We now show that the order of the finite differences can be reduced in some cases.
\begin{proof}[Proof of Lemma \ref{mton}]
	Fix $s\in(0,1)$ and $x\in U$.  Let
	\begin{align*}
	P_a:=\sum_{k=-a}^a (-1)^k { \binom{2a}{a-k}} k^{2s}\qquad \text{ for }a\in\N\text{ with }a>s
	\end{align*}
	and let $n,m\in\N$ such that $s<n<m$. Let $u\in C^{2s+\beta}(U)\cap \cL^1_{s}$ for some $\beta\in(0,1)$, then, by changing variables,
	\begin{align*}
	P_n\int_{\R^N}\frac{\delta_m u(x,y)}{|y|^{N+2s}}\ dy
	&=\sum_{j=-n}^n(-1)^j { \binom{2n}{n-j}} \int_{\R^N}\frac{ \delta_m u(x,jy)}{|y|^{N+2s}}\ dy\\
	&= \int_{\R^N}\frac{ \sum_{j=-n}^n\sum_{k=-m}^m(-1)^k { \binom{2m}{m-k}} (-1)^j { \binom{2n}{n-j}} u(x+kjy)}{|y|^{N+2s}}\ dy\\
	&=\sum_{k=-m}^m(-1)^k { \binom{2m}{m-k}} \int_{\R^N}\frac{ \delta_n u(x,ky)}{|y|^{N+2s}}\ dy=P_m\int_{\R^N}\frac{\delta_n u(x,y)}{|y|^{N+2s}}\ dy.
	\end{align*}
	Since $c_{N,\m,s}\frac{P_{m}}{P_{n}}=c_{N,n,s}$, we have that
	\begin{align*}
	L_{m,s}u(x)=\frac{c_{N,\m,s}}{2}\int_{\R^N} \frac{\delta_\m u(x,y)}{|y|^{N+2s}} \ dy=\frac{c_{N,\m,s}}{2}\frac{P_{m}}{P_{n}}\int_{\R^N} \frac{\delta_n u(x,y)}{|y|^{N+2s}} \ dy=L_{n,s}u(x),
	\end{align*}
	as claimed.
\end{proof}

\subsection{Integration by parts formulas}

\begin{proof}[Proof of Lemma \ref{ibyp}]
Let $m\in \N$, $\beta\in(0,1)$, $s\in(0,m)$, $U\subset\R^N$ open, $u\in C^{2s+\beta}(U)\cap \cL^1_s$, and $\varphi\in \cC^\infty_c(U)$. By Fubini's theorem and changes of variables,
 \begin{align*}
\int_{\R^N} L_{m,s}u(x)\varphi(x)&\ dx=\frac{c_{N,\m,s}}{2}\int_{\R^N}\sum_{k=-\m}^\m (-1)^k { \binom{2m}{m-k}}|y|^{-N-2s}\int_{\R^N} u(x+ky)\varphi(x) \ dx\ dy\\
&=\frac{c_{N,\m,s}}{2}\int_{\R^N}\sum_{k=-\m}^\m (-1)^k { \binom{2m}{m-k}}|y|^{-N-2s}\int_{\R^N} u(x)\varphi(x+ky) \ dx\ dy\\
&=\frac{c_{N,\m,s}}{2}\int_{\R^N}u(x)\int_{\R^N} \frac{\sum_{k=-\m}^\m (-1)^k { \binom{2m}{m-k}} \varphi(x+ky)}{|y|^{N+2s}} \ dy\ dx=\int_{\R^N} u(x)L_{m,s} \varphi(x)\ dx.
\end{align*}
\end{proof}

We now present a simple integration by parts argument, see \cite[Lemma 2]{TD16} for a similar result.
\begin{lemma}\label{dibyp}
 Let $s>0$, $n,m\in\N$, $s<m$, and $u,v\in \cC^\infty_c(\R^N)$. Then,
 \begin{align}\label{eq:dib:1}
  \int_{\R^N}\int_{\R^N}\frac{\delta_{n}u(x,y)\delta_mv(x,y)}{|y|^{N+2s}}\ dx\, dy
  =\int_{\R^N}\int_{\R^N}\frac{\delta_{n+m}u(x,y)v(x)}{|y|^{N+2s}}\ dx\, dy.
 \end{align}
 Moreover, 
 \begin{align}\label{eq:dib:2}
  \int_{\R^N}\int_{\R^N}\frac{u(x)\delta_1v(x,y)}{|y|^{N+2s}}\ dx\, dy
  =\int_{\R^N}\int_{\R^N}\frac{(u(x)-u(x+y))(v(x)-v(x+y))}{|y|^{N+2s}}\ dx\, dy.
 \end{align}
\end{lemma}
\begin{proof}
Observe that, by a change of variables
\begin{align}
  \int_{\R^N}\int_{\R^N}&\frac{u(x)\delta_1v(x,y)}{|y|^{N+2s}}\ dx\, dy=\int_{\R^N}\int_{\R^N}\frac{u(x)(2v(x)-v(x-y)-v(x+y))}{|y|^{N+2s}}\ dx\, dy\nonumber\\
  &=\int_{\R^N}|y|^{-N-2s}\int_{\R^N}2v(x)u(x)\ dx-\int_{\R^N}u(x)v(x-y)\ dx-\int_{\R^N}u(x)v(x+y)\ dx\, dy\label{eq:dib:3}\\
  &=\int_{\R^N}\int_{\R^N}\frac{(2u(x)-u(x-y)-u(x+y))v(x)}{|y|^{N+2s}}\ dx\, dy=\int_{\R^N}\int_{\R^N}\frac{\delta_1u(x,y)v(x)}{|y|^{N+2s}}\ dx\, dy,\nonumber
  \end{align}
and \eqref{eq:dib:1} now follows from \eqref{iteration1}. Note that \eqref{eq:dib:2} also follows from \eqref{eq:dib:3} by a suitable change of variables.
 \end{proof}

\begin{lemma}[Lemma 2.4 in \cite{AJS17}]\label{ibyp:ours}
Let $U\subset \R^N$ be an open bounded set with Lipschitz boundary, $\beta,\sigma\in(0,1)$, $n\in\N_0$, $s=n+\sigma$, $u\in C^{2s+\beta}(U)\cap\cL^1_\sigma$. Then
\begin{align*}
\int_{\R^N} u\, (-\Delta)^m(-\Delta)^\sigma \varphi \ dx = \int_{\R^N} \varphi\,(-\Delta)^{m}(-\Delta)^\sigma u\ dx \qquad \text{for all $\varphi\in \cC^\infty_c(U)$.}
\end{align*}
Moreover, if $u\in \cH_0^s(U)$ then 
\begin{align*}
\int_{\R^N} u\, (-\Delta)^m(-\Delta)^\sigma\varphi\ dx = \cE_s(u,\varphi)\qquad \text{ for all }\varphi\in \cC_c^\infty(U).
\end{align*}
\end{lemma}

\subsection{Known identities}
We use the following definitions. Let $\cL(f)$ denote the Laplace transform and $\Gamma$ the Gamma function, then
\begin{align}
 \Gamma(\rho):=\int_0^\infty w^{x-1}e^{-w}\ dw\quad \text{ and }\quad \cL(f)(\rho):=\int_0^\infty f(t)e^{-\rho t} dt,\qquad \rho>0\label{Gamma:def}
\end{align}
for $f\in L^\infty(\R^N)$.   We also use the next known identities, for $s,a,b>0$, $N,m,n\in\N$, $\sigma\in(0,1)$, and $s=n+\sigma$, we have that
\begin{align}
 \cL(\sin^{2\m})(a)& = \frac{(2\m)!}{a \prod_{k=1}^{\m}(4k^2+a^2)},\label{Lapl:id}\\
 \int_0^{\infty} \frac{\rho^{\sigma-1}}{1+\rho}\ d\rho&=\frac{\pi}{\sin(\pi \sigma)}=\frac{(-1)^{n-1}\pi}{\sin(\pi s)}=(-1)^{n-1}\Gamma(s)\Gamma(1-s),\label{gamma:id:2}\\
 \int_{0}^{\infty}\frac{\rho^{N-2}}{(1+\rho^2)^{\frac{N+2s}{2}}}\ d\rho&=\frac{\Gamma(\frac{N-1}{2})\Gamma(s+\frac{1}{2})}{2 \Gamma(\frac{N}{2}+s)},\label{beta:id}\\
 \int_0^{2\pi}\sin^{2a}(t)\cos^{2b}(t)\ dt&=2\frac{\Gamma(\frac{1}{2}+a)\Gamma(\frac{1}{2}+b)}{\Gamma(1+a+b)}
 =\frac{2^{1-2(a+b)}\pi \Gamma(2a+1)\Gamma(2b+1)}{\Gamma(a+1)\Gamma(b+1)\Gamma(1+a+b)},\label{betaid}\\
 \int_0^\pi \sin^{2a}(t)\cos^{2b}(t)\ dt&=\frac{\Gamma(\frac{1}{2}+a)\Gamma(\frac{1}{2}+b)}{\Gamma(1+a+b)}
=\frac{\Gamma(\frac{1}{2}+a)4^{-b}\pi^\frac{1}{2}(2b)!}{\Gamma(1+a+b)b!},\qquad b\in\N_0,\label{betaid2}\\
a\Gamma(a)=\Gamma(a+1),\quad &\text{ and }\quad \Gamma(\frac{1}{2}+a)=2^{1-2a}\pi^\frac{1}{2}\frac{\Gamma(2a)}{\Gamma(a)}\label{Gammaid}.
\end{align}
Identities \eqref{beta:id} and \eqref{gamma:id:2} are particular cases of \cite[page 10, formula (16)]{EMOT81}, while \eqref{Gammaid} can be retrieved in \cite[page 3, formula (1) and page 5, formula (15)]{EMOT81}. Identity \eqref{betaid2} is a consequence of \eqref{betaid}, which can be deduced from \cite[page 10, formula (17)]{EMOT81} after a change of variable. Finally, for identity \eqref{Lapl:id} we refer to \cite[page 150, formula (3)]{EMOT54}.

\section{Equivalence of evaluations} \label{Equiv:sec}
\begin{proof}[Proof of Theorem \ref{new:main:thm}]
Let $n\in\N$, $\sigma\in(0,1)$, and $s=m+\sigma$.  Observe that, for $y\in\R^N\backslash\{0\}$,
\begin{align*}
-c_{N,1,\sigma}(-\Delta)^n &{|y|^{-N-2\sigma}}=c_{N,1,\sigma}(-1)^{n+1}\prod\limits_{i=0}^{n-1}(N+2\sigma+2i)(2\sigma+2(i+1)){|y|^{-N-2\sigma-2n}}\\
&=\ (-1)^{n}\frac{2^{2\sigma}\,\Gamma(N/2+\sigma)}{\pi^{N/2}\,\Gamma(-\sigma)}\prod\limits_{i=0}^{n-1}(N+2\sigma+2i)(2\sigma+2(i+1))|y|^{-N-2\sigma-2n}\\
&= (-1)^{n}\frac{2^{2s}\,\Gamma(N/2+\sigma)}{\pi^{N/2}\,\Gamma(-\sigma)}\prod\limits_{i=0}^{n-1}(N/2+\sigma+i)(\sigma+i+1) |y|^{-N-2\sigma-2n}\\
&= \frac{2^{2s}\,\Gamma(N/2+s)}{\pi^{N/2}\,\Gamma(-\sigma)}\prod\limits_{i=0}^{n-1}(-\sigma-i-1) |y|^{-N-2\sigma-2n}
= \frac{4^{s}\,\Gamma(\frac{N}{2}+s)}{\pi^{N/2}\,\Gamma(-s)}|y|^{-N-2\sigma-2n}.
\end{align*}
Therefore,
\begin{align}\label{art:1}
 |y|^{-N-2\sigma-2n}=\frac{-c_{N,1,\sigma}\pi^{N/2}\,\Gamma(-s)}{4^{s}\,\Gamma(\frac{N}{2}+s)}(-\Delta)^n|y|^{-N-2\sigma}\qquad \text{ for }y\neq 0.
\end{align}
Let $U\subset \R^N$ open, $u\in C^{2s+\beta}(U)$, $(-\Delta)^n u\in\cL^1_{\sigma}$, $(-\Delta)^iu\in\cL^1_{s-i-\frac{1}{2}}$, and 
$|\nabla(-\Delta)^i u|\in\cL^1_{s-i-1}$ for $i\in\{0,\ldots n-1\}$.
Then, there is $r_j\to\infty$ as $j\to\infty$ such that
\begin{equation}\label{boundaries}
\begin{aligned}
 \lim_{j\to\infty}\int_{\partial B_{r_j}(0)}|(-\Delta)^i u(y)| \partial_\nu (-\Delta)^{n-i-1}|y|^{-N-2\sigma}\ dt&=k_{i}\lim_{j\to\infty}\int_{\partial B_{r_j}(0)}\frac{|(-\Delta)^i u(y)|}{|y|^{N+2s-2i-1}}\ dt =0,\\
\lim_{j\to\infty}\int_{\partial B_{r_j}(0)}|\partial_\nu(-\Delta)^i u(y)| (-\Delta)^{n-i-1}|y|^{-N-2\sigma}\ dt&=k_{i}'\lim_{j\to\infty}\int_{\partial B_{r_j}(0)}\frac{|\partial_\nu(-\Delta)^i u(y)|}{|y|^{N+2s-2i-2}}\ dt =0
\end{aligned}
\end{equation}
for $i\in\{0,\ldots, n-1\}$ and some positive constants $k_{i}$ and $k'_{i}$. These limits ensure that the boundary terms from the integration by parts performed below vanish. To shorten notation, let $P:= \sum_{k=1}^\m (-1)^k { \binom{2m}{m-k}}k^{2s}$.  Then, using \eqref{art:1}, \eqref{calc}, integration by parts, \eqref{boundaries}, and change of variables, we have, for $x\in U$,
\begin{align*}
L_{m,s}u(x)&=\frac{c_{N,m,s}}{2}\int_{\R^N}\frac{\delta_m u(x,y)}{|y|^{N+2s}}\ dy=\frac{4^{s}\Gamma(\frac{N}{2}+s)}{ \pi^{\frac{N}{2}}\Gamma(-s) 2P}\int_{\R^N}\frac{\delta_m u(x,y)}{|y|^{N+2s}}\ dy\\
&=-\frac{c_{N,1,\sigma}}{2P}\lim_{j\to\infty}\int_{B_{r_j}(0)}\delta_m u(x,y)(-\Delta)^n|y|^{-N-2\sigma}\ dy\\
&=-\frac{c_{N,1,\sigma}}{P}\lim_{j\to\infty}\int_{B_{r_j}(0)}(-\Delta)_y^n\delta_m u(x,y)|y|^{-N-2\sigma}\ dy\\
&=-\frac{c_{N,1,\sigma}}{2P}\int_{\R^N} \sum_{k=-\m}^\m (-1)^k { \binom{2m}{m-k}}k^{2n} \frac{(-\Delta)^nu(x+ky)}{|y|^{N+2\sigma}}\ dy\\
&=-\frac{c_{N,1,\sigma}}{2P}\int_{\R^N} \sum_{k=-\m}^\m (-1)^k { \binom{2m}{m-k}}k^{2n} \frac{(-\Delta)^nu(x+ky)-(-\Delta)^nu(x)}{|y|^{N+2\sigma}}\ dy\\
&=-\frac{c_{N,1,\sigma}}{2P}\int_{\R^N} \sum_{k=-\m}^\m (-1)^k { \binom{2m}{m-k}}k^{2s} \frac{(-\Delta)^nu(x+y)-(-\Delta)^nu(x)}{|y|^{N+2\sigma}}\ dy\\
&=c_{N,1,\sigma}\int_{\R^N} \frac{(-\Delta)^nu(x)-(-\Delta)^nu(x+y)}{|y|^{N+2\sigma}}\ dy = (-\Delta)^\sigma(-\Delta)^nu(x),
  \end{align*}
  as claimed.
\end{proof}

Before we proceed to the proof of Corollary \ref{e:d:cor}, we recall a result on interchange of derivatives.
\begin{prop}[Proposition B.2 in \cite{AJS16}]\label{interchange}
Let $\Omega\subset \R^N$ open, $\sigma\in (0,1)$, and $u\in C^{3}(\Omega)\cap\cL^1_{\sigma}\cap W^{1,1}_{loc}(\mathbb R^N)$. If $\partial_1 u\in \cL^1_{\sigma}$, then $\partial_1(-\Delta)^\sigma u(x) = (-\Delta)^\sigma \partial_1u(x)$ pointwisely for all $x\in \Omega$.
\end{prop}

\begin{proof}[Proof of Corollary \ref{e:d:cor}] 
Let $u\in \cL^1_s\cap C^{2s+\beta}(U)$ and  $\varphi\in C^\infty_c(\R^N)$, then, by Lemmas \ref{ibyp} and \ref{new:main:thm},
\begin{align}\label{basis}
 \int_{\R^N} L_{m,s} u(x)\varphi(x)\ dx
 =\int_{\R^N} u(x)  L_{m,s} \varphi(x)\ dx
 =\int_{\R^N} u(x) (-\Delta)^\sigma(-\Delta)^n \varphi(x)\ dx=:E.
\end{align}
Claim $(a)$ now follows from \eqref{basis}, Lemma \ref{ibyp:ours}, a standard integration by parts using that $\varphi$ has compact support, and the fundamental lemma of calculations of variations. For claim $(b)$, since $u\in\cH^s_0(U)$, we have, by Proposition \ref{interchange} and Lemma \ref{ibyp:ours}, that
\begin{align*}
E=\int_{\R^N} u(x) (-\Delta)^\frac{n}{2}(-\Delta)^\sigma(-\Delta)^\frac{n}{2}\varphi(x)\ dx=\int_{\R^N} (-\Delta)^\frac{n}{2}(-\Delta)^\sigma(-\Delta)^\frac{n}{2} u(x) \varphi(x)\ dx
\end{align*}
and the claim follows similarly from \eqref{basis}. Claim $(c)$ can be argued analogously.
\end{proof}

\subsection{Asymptotic analysis}

We now study the asymptotic behaviour of \eqref{cNms:def}.

\begin{lemma}\label{a:lemma}
Let $N,m\in \N$, $s\in(0,m)$, and $c_{N,m,s}$ as in \eqref{cNms:def}. Then $s\mapsto c_{N,m,s}>0$ is continuous in $(0,m)$. Moreover,
\begin{align*}
\lim_{s\to 0^+}\frac{c_{N,m,s}}{s}&=\frac{2(m!)^2\Gamma(\frac{N}{2})}{(2m)!\pi^{\frac{N}{2}}}>0\qquad
\text{ and }\qquad \lim_{s\to m^-}\frac{c_{N,m,s}}{m-s}=\frac{2^{2m+1}m!\Gamma(\frac{N}{2}+m)}{(2m)!\pi^{\frac{N}{2}}} >0.
\end{align*}	 
\end{lemma}
\begin{proof}
Let $N,m,n\in \N$, $n<m$, $s\in(0,m)\backslash \N$, and $P(s):=\sum_{k=1}^{\m}(-1)^{k}{ \binom{2m}{m-k}} k^{2s}$, then
\begin{align*}
\frac{1}{c_{N,m,s}}= \frac{ \pi^{\frac{N}{2}}\Gamma(n-s+1) }{4^{s}\Gamma(\frac{N}{2}+s)}\Big(\prod_{i=0}^{n-1}\frac{1}{(i-s)}\Big)\frac{P(s)-P(n)}{n-s},
\end{align*}
because $P(n)=0$, by \eqref{calc}. Then, $\lim_{s\to n}c_{N,m,s}=\frac{4^{n}\Gamma(\frac{N}{2}+n)}{\pi^{\frac{N}{2}}}(-1)^nn!(-P'(n))^{-1}=c_{N,m,n}$. Since the continuity of $s\mapsto c_{N,m,s}>0$ is clear in $(0,m)\backslash \N$ the first assertion follows. Furthermore, by \eqref{calc},
 \begin{align*}
  \lim_{s\to 0^+}\frac{c_{N,m,s}}{s}&=\lim_{s\to 0^+}\frac{2^{2s}\Gamma(\frac{N}{2}+s)}{ \pi^{\frac{N}{2}}\Gamma(1-s) }\Big(\sum_{k=1}^{\m}(-1)^{k+1}{ \binom{2m}{m-k}} k^{2s}\Big)^{-1}
  =\frac{\Gamma(\frac{N}{2})}{ \pi^{\frac{N}{2}}}\frac{2(m!)^2}{(2m)!}.
 \end{align*}
 On the other hand, by Lemma \ref{calc:lemma}, $\sum_{k=1}^{\m}(-1)^{k+1}{ \binom{2m}{m-k}} k^{2m}=\frac{(-1)^{m-1}(2m)!}{2},$ and therefore, 
\begin{align*}
\lim_{s\to m^-}\frac{c_{N,m,s}}{m-s}&=\lim_{s\to m^-}\frac{1}{m-s}\frac{2^{2s}\Gamma(\frac{N}{2}+s)}{ \pi^{\frac{N}{2}}\Gamma(-s) }\Big(\sum_{k=1}^{\m}(-1)^{k}{ \binom{2m}{m-k}} k^{2s}\Big)^{-1}\\
&=\frac{2^{2m+1}\Gamma(\frac{N}{2}+m)}{ \pi^{\frac{N}{2}}(-1)^{m-1}(2m)!}
\lim_{s\to m^-}\frac{1}{(m-s)\Gamma(-s)}
=\frac{2^{2m+1}\Gamma(\frac{N}{2}+m)}{ \pi^{\frac{N}{2}}(-1)^{m-1}(2m)!}
\lim_{s\to m^-}\frac{\prod_{j=0}^{m-1}(j-s)}{\Gamma(m+1-s)}\\
&=\frac{2^{2m+1}\Gamma(\frac{N}{2}+m)}{ \pi^{\frac{N}{2}}(-1)^{m-1}(2m)!}
\lim_{s\to m^-}\prod_{j=0}^{m-1}(j-m)
=\frac{2^{2m+1}\Gamma(\frac{N}{2}+m)}{ \pi^{\frac{N}{2}}(2m)!}
m!
\end{align*}
\end{proof}

\begin{lemma}
Let $N,m\in\N$, $s\in(0,m)$, and $\alpha\in\N^N_{0}$ such that $|\alpha|=m$. Then,
\begin{align}\label{alpha:int}
 2(m-s)\int_{B_1}\frac{y^{2\alpha}}{|y|^{N+2s}}\ dy=\frac{(2\alpha)!}{\alpha!}\frac{\pi^{\frac{N}{2}}}{2^{2m-1}\Gamma(\frac{N}{2}+m)}
\end{align}
In particular,
\begin{align}\label{alpha:eq}
\frac{(-1)^m(2m)!}{(2\alpha)!}
\lim_{s\to m^-}\frac{c_{N,m,s}}{2}
\int_{B_1}\frac{y^{2\alpha}}{|y|^{N+2s}}\ dy
 =(-1)^m\frac{m!}{\alpha!}. 
\end{align}
\end{lemma}
\begin{proof}

If $N=1$, then $\alpha=m$, $ \int_{-1}^{1}|y|^{1+2s-2m}\ dy=(m-s)^{-1}$ and \eqref{alpha:int} follows because, by \eqref{Gammaid},
$\frac{(2m)!}{m!}\frac{\pi^{\frac{1}{2}}}{2^{2m}}\frac{1}{\Gamma(\frac{1}{2}+m)}=1$.  If $N=2$ we use polar coordinates, \emph{i.e.}, $y_1=r\cos\theta$ and $y_2=r\sin\theta$ for $r\in(0,1)$ and $\theta\in (0,2\pi)$. Let $\alpha\in\N_0^2$ such that $\alpha_1+\alpha_2=m$, then, by \eqref{beta:id}, \eqref{Gammaid},
\begin{align}
 &\int_{B_1}\frac{y^{2\alpha}}{|y|^{2+2s}}\ dy=\frac{1}{2(m-s)}\int_0^{2\pi}\cos^{2\alpha_1}\theta\sin^{2\alpha_2}\theta \ d\theta=2^{-2m}\pi\frac{(2\alpha)!}{(m-s)(\alpha)!m!}
 \label{Dim2}.
\end{align}
Observe that \eqref{alpha:int} follows from \eqref{Dim2}.

For the general case $N\geq 3$ we use spherical coordinates, \emph{i.e.}, $y_{i}=r\cos \theta_{i} \prod_{l=1}^{i-1}\sin\theta_l$ for $i\in\{1,\ldots,N-1\}$ and $y_{N}=r \prod_{l=1}^{N-1}\sin\theta_l$, where 
$r>0,$ $\theta_1,\ldots,\theta_{N-2}\in(0,\pi),$ $\theta_{N-1}\in(0,2\pi)$, and the associated Jacobian is 
$J(r,\theta_1,\ldots,\theta_{N-1})=r^{N-1}\prod_{j=1}^{N-2}\sin^{N-1-j}\theta_j.$

Let $\alpha\in\N^N_{0}$ such that $|\alpha|=m$ and for $i=0,1,\ldots, N-1$ let $S_i:=\sum_{j=i+1}^N \alpha_j$, then,
\begin{align*}
 y^{2\alpha} J(r,\theta_1,\ldots,\theta_{N-1})=&r^{N-1+2m}\sin^{2S_1+N-2}\theta_1\cos^{2\alpha_1}\theta_1\sin^{2S_2+N-3}\theta_2\cos^{2\alpha_2}\theta_2\\
 &\qquad\qquad\qquad\qquad\qquad\qquad \cdots \sin^{2S_{N-1}}\theta_{N-1}\cos^{2\alpha_{N-1}}\theta_{N-1}.
 \end{align*}
Thus, by \eqref{betaid} and \eqref{betaid2},
\begin{align*}
 2(m-s)\int_{B_1}&\frac{y^{2\alpha}}{|y|^{N+2s}}\ dy
 =2(m-s)\int_{B_1}\frac{\prod_{i=1}^N y_i^{2\alpha_i}}{|y|^{N+2s}}\ dy\\
 &=\int_0^{2\pi} 
 \sin^{2\alpha_{N}}\theta_{N-1}\cos^{2\alpha_{N-1}}\theta_{N-1}
 d\theta_{N-1}
 \prod_{i=1}^{N-2}\int_0^{\pi}
 \sin^{2(S_{i}+\frac{N-i-1}{2})}\theta_i\cos^{2\alpha_i}\theta_i\ d\theta_i\\
&=\frac{2^{1-2(\alpha_{N}+\alpha_{N-1})}\pi (2\alpha_{N})!(2\alpha_{N-1})!}{(\alpha_{N})!(\alpha_{N-1})!(\alpha_{N}+\alpha_{N-1})!}
\prod_{i=1}^{N-2} 
\frac{\Gamma(\frac{1}{2}+S_{i}+\frac{N-i-1}{2})4^{-\alpha_i}\pi^\frac{1}{2}(2\alpha_i)!}{\Gamma(1+S_{i}+\frac{N-i-1}{2}+\alpha_i)\alpha_i!}\\
&=\frac{\pi^{\frac{N}{2}}(2\alpha)!}{2^{2m-1}\alpha!}\frac{1}{(\alpha_{N}+\alpha_{N-1})!}
\prod_{i=1}^{N-2} 
\frac{\Gamma(\frac{N-i}{2}+S_i)}{\Gamma(\frac{N-i+1}{2}+S_{i-1})}\\
&=\frac{\pi^{\frac{N}{2}}(2\alpha)!}{2^{2m-1}\alpha!}\frac{1}{(\alpha_{N}+\alpha_{N-1})!}
\frac{\Gamma(\frac{N-(N-2)}{2}+S_{N-2})}{\Gamma(\frac{N}{2}+S_{0})}
=\frac{\pi^{\frac{N}{2}}(2\alpha)!}{2^{2m-1}\alpha!\Gamma(\frac{N}{2}+m)},
\end{align*}
and \eqref{alpha:int} follows.  Equation \eqref{alpha:eq} follows from \eqref{alpha:int} and Lemma \ref{a:lemma} by direct substitution.
\end{proof}

We now proceed with our asymptotic analysis on the operator $L_{m,s}$ as $s$ approaches $m$ from below. For a similar asymptotic study in the case $s\in(0,1)$, see \cite[Proposition 4.4]{NPV11}.

\begin{prop}\label{lim:prom}
Let $m\in \N$, $\eta,\beta\in(0,1)$, $U\subset \R^N$ open, and $x\in U$. Then 
\begin{align}
\lim_{s\to0^+}L_{m,s}u(x)&=u(x) \quad\quad\quad\quad\text{ for all $u\in C^{\beta}(U)\cap L^{\infty}(\R^N)$,}\label{fi:cl}\\
\lim_{s\to m^-}L_{m,s}u(x)&=(-\Delta)^mu(x) \quad\text{ for all $u\in \cC^{2m}(U)\cap \cL^1_{m-\eta}$.}\label{sc:cl}
\end{align}
\end{prop}
\begin{proof}
By Lemma \ref{mton}, we have that $L_{m,s}$ reduces to $L_{1,s}$ if $s\in(0,1)$, and therefore \eqref{fi:cl} follows from \cite[Proposition 4.4]{NPV11}. For \eqref{sc:cl}, let $u\in \cC^{2m}(U)\cap \cL^1_{m-\eta}$, $x\in U$, $\eps>0$, and let $\rho(U,x,m,u,\varepsilon)=\rho\in(0,1)$ be the constant given by Lemma \ref{mte:lemma:2}. Let $s\in(m-\eta,m)$. We collect first some useful facts. Since $u\in\cL^1_{m-\eta}\subset\cL^1_s$,
\begin{align*}
\Big| \int_{\R^N\backslash B_\rho} \frac{\delta_m u(x,y)}{|y|^{N+2s}}\ dy \Big|
 \leq \sum_{k=-m}^m { \binom{2m}{m-k}}\int_{\R^N\backslash B_\rho} \frac{|u(x+ky)|}{|y|^{N+2s}}\ dy<\infty.
\end{align*}
Then, by Lemma \ref{a:lemma}, $\lim_{s\to m^-} c_{N,m,s}=0$ and therefore
\begin{align}\label{z:out}
 \lim_{s\to m^-}\frac{c_{N,m,s}}{2} \int_{\R^N\backslash B_{{\rho}}} \frac{\delta_m u(x,y)}{|y|^{N+2s}}\ dy=0.
\end{align}
On the other hand, by Lemma \ref{mte:lemma:2}, 
 \begin{align}\label{Fl1}
 \int_{B_{\rho}} \frac{\delta_m u(x,y)}{|y|^{N+2s}}\ dy=
 \int_{B_{\rho}} \left(\frac{R(x,y)}{|y|^{2s+N}}+\sum_{|\alpha|=2m}\frac{(-1)^m(2m)!}{\alpha!} \partial^\alpha u(x)\frac{y^\alpha}{|y|^{N+2s}}\right) dy,
\end{align}
where $|R(x,y)|\leq C \eps|y|^{2m}$ for all $x,y\in B_{\rho}$ and for some $C(N,m,u)=C>0$.  In particular,
\begin{align}\label{Fl2}
\left|\int_{B_{{\rho}}} \frac{R(x,y)}{|y|^{2s+N}}\;dy\right|\leq C_1\frac{\eps\rho^{2(m-s)}}{m-s}\qquad \text{ for some }C_1(N,m,u)>0.
\end{align}

Furthermore, note that, if $\alpha\in\N^N_{0}$ is such that $\alpha_i\neq 0$ is odd for some $i\in\{1,\ldots,N\}$, then
\begin{align}\label{symzero}
\int_{B_{\rho}} \frac{\prod_{i=1}^N y^{\alpha_i}_i}{|y|^{N+2s}}\ dy
=-\int_{B_{\rho}} \frac{\prod_{i=1}^N \widetilde y^{\alpha_i}_i}{|\widetilde y|^{N+2s}}\ d\widetilde y,
\qquad \text{\emph{i.e.}} \quad \int_{B_{\rho}} \frac{y^{\alpha}}{|y|^{N+2s}}\ dy=0.
\end{align}

Finally, by the multinomial theorem,
\begin{align}\label{pl}
(-1)^m\sum_{|\alpha|=m}\frac{m!}{\alpha!}\partial^{2\alpha} u(x)=(-1)^m(\sum_{i=1}^N \partial_{ii})^m u(x)=(-\Delta)^m u(x).
\end{align}
Therefore, by \eqref{pl}, \eqref{alpha:eq}, a change of variables, \eqref{symzero}, \eqref{Fl1}, \eqref{z:out} and \eqref{Fl2},
\begin{align*}
& \left|\lim_{s\to m^-}L_{m,s}u(x)-(-\Delta)^m u(x)\right|
=\left|\lim_{s\to m^-}L_{m,s}u(x)
 - \sum_{|\alpha|=m}(-1)^m\frac{m!}{\alpha!}\partial^{2\alpha} u(x) \right|\\
&=\left|\lim_{s\to m^-}L_{m,s}u(x)
 -\sum_{|\alpha|=m}\partial^{2\alpha} u(x)\frac{(-1)^m (2m)!}{(2\alpha)!} \lim_{s\to m^-}\frac{c_{N,m,s}}{2}\rho^{2m-2s}\int_{B_1}\frac{y^{2\alpha}}{|y|^{N+2s}}\ dy\right|\\
&=\left|\lim_{s\to m^-}L_{m,s}u(x)
 -\lim_{s\to m^-}\frac{c_{N,m,s}}{2}\sum_{|\alpha|=m}\partial^{2\alpha} u(x)\frac{(-1)^m (2m)!}{(2\alpha)!} \int_{B_{\rho}}\frac{y^{2\alpha}}{|y|^{N+2s}}\ dy\right|\\
 &=\left|\lim_{s\to m^-}L_{m,s}u(x)
 -\lim_{s\to m^-}\frac{c_{N,m,s}}{2}\sum_{|\alpha|=2m}\partial^{\alpha} u(x)\frac{(-1)^m (2m)!}{(\alpha)!}\int_{B_{\rho}}\frac{y^{\alpha}}{|y|^{N+2s}}\ dy\right|\\
 &=\left|\lim_{s\to m^-}L_{m,s}u(x)
 -\lim_{s\to m^-}\frac{c_{N,m,s}}{2}
 (\int_{B_{\rho}} \frac{\delta_m u(x,y)}{|y|^{N+2s}}\ dy-\int_{B_{\rho}} \frac{R(x,y)}{|y|^{2s+N}}\ dy)\right|\\
 &=\left|\lim_{s\to m^-}\frac{c_{N,m,s}}{2}\int_{B_{\rho}} \frac{R(x,y)}{|y|^{2s+N}}\ dy\right|\leq 
 \frac{C_1}{2}\eps \lim_{s\to m^-}\frac{c_{N,m,s}}{m-s}\rho^{2(m-s)}
 =\frac{C_1}{2}\frac{2^{2m+1}m!\Gamma(\frac{N}{2}+m)}{(2m)!\pi^{\frac{N}{2}}}\eps.
 \end{align*}
The result follows by letting $\eps\to 0$.
\end{proof}

\begin{proof}[Proof of Theorem \ref{van:cor}]
Let $m,n\in \N$ with $n<m$. The limits \eqref{asymp} follow from Proposition \ref{lim:prom}. Furthermore, by Lemma \ref{a:lemma}, $s\mapsto c_{N,m,s}>0$ is continuous in $(0,m)$, and therefore
	\begin{align*}
	L_{m,n}u(x)=\lim_{s\to n^-} L_{m,s}u(x)=\lim_{s\to n^-} L_{n,s}u(x)=(-\Delta)^n u(x)\qquad \text{ for all }x\in U,
	\end{align*}
by Lebesgue dominated convergence, Lemma \ref{mton}, and Proposition \ref{lim:prom}. This ends the proof.
\end{proof}

\subsection{The bilinear form}

\begin{proof}[Proof of Theorem \ref{bilinear}]
Let $m,n\in \N$, $s>0$, $s<n\leq 2m$, and $\varphi,\psi\in \cC^\infty_c(\R^N)$. By Lemmas \ref{mton}, \ref{dibyp}, and Fubini's theorem,
\begin{align*}
 \cE_{2m,s}(\varphi,\psi)&=\frac{c_{N,2m,s}}{2}\int_{\R^N}\int_{\R^N}\frac{\delta_{m}\varphi(x,y) \delta_{m}\psi(x,y)}{|y|^{N+2s}}\ dxdy
 =\int_{\R^N}\frac{c_{N,2m,s}}{2}\int_{\R^N}\frac{\delta_{2m}\varphi(x,y)}{|y|^{N+2s}}\ dy\ \psi(x)dx\\&
 =\int_{\R^N}L_{2m,s}\varphi(x)\psi(x)\ dx=\int_{\R^N}L_{n,s}\varphi(x)\psi(x)\ dx.
\end{align*}

By Theorem \ref{new:main:thm}, Lemma \ref{ibyp:ours}, and Proposition \ref{interchange} we have that
\begin{align}\label{smooth}
 \cE_{2m,s}(\varphi,\psi)=\cE_{s}(\varphi,\psi).
 \end{align}

Now, let $s\in(i,i+1)$ for some $i\in\N$, $i\leq 2m-1$, $u,v\in H^s(\R^N)$ and, for $j\in\N$, let $u_j,v_j\in \cC^\infty_c(\R^N)$ be such that $u_j\to u$ and $v_j\to v$ in $H^s(\R^N)$ as $j\to\infty$.  By \eqref{calc}, \eqref{mte:eq}, a change of variables, and Fatou's Lemma,
\begin{align*}
\cE_{2m,s}(u,u)&\leq\liminf_{j\to\infty}\cE_{2m,s}(u_j,u_j)=\frac{c_{N,2m,s}}{2}\liminf_{j\to\infty}\int_{\R^N}\int_{\R^N}\frac{|\delta_{m}u_j(x,y)|^2}{|y|^{N+2s}}\ dxdy\nonumber\\
&=\frac{c_{N,2m,s}}{2}\liminf_{j\to\infty}\int_{\R^N}\int_{\R^N}\frac{\left|\sum_{k=-\m}^\m (-1)^k { \binom{2m}{m-k}} u_j(x+ky)\right|^2}{|y|^{N+2s}}\ dxdy\nonumber\\ 
&\leq c_{N,2m,s}\liminf_{j\to\infty}\int_{\R^N}\int_{\R^N}\frac{\left(\sum_{|\alpha|=i}\frac{1}{\alpha!} \sum_{k=1}^\m { \binom{2m}{m-k}}k^{i} |\partial^\alpha u_j(x+k y\theta_j)-\partial^\alpha u_j(x)|\right)^2}{|y|^{N+2(s-i)}}\ dxdy\nonumber\\
&\leq C\liminf_{j\to\infty}\sum_{|\alpha|=i}\ \int_{\R^N}\int_{\R^N}\frac{|\partial^\alpha u_j(x+y)-\partial^\alpha u_j(x)|^2}{|y|^{N+2(s-i)}}\ dxdy\\
&\leq C \liminf_{j\to\infty}\|u_j\|^2_{H^s(\R^N)}=C\|u\|^2_{H^s(\R^N)},
\end{align*}
where $\theta_j\in[0,1]$ is given by Lemma \ref{mte:lemma} and $C(N,m,s)=C>0$. Therefore, by H\"{o}lder inequality,
\begin{align*}
|\cE_{2m,s}(u,v)|&=\frac{c_{N,2m,s}}{2}\left|\int_{\R^N}\int_{\R^N}\frac{\delta_{m}u(x,y)}{|y|^{N/2+s}}\: \frac{\delta_{m}v(x,y)}{|y|^{N/2+s}}\ dxdy\right| \\
& \leq\ \frac{c_{N,2m,s}}{2}\left(\int_{\R^N}\int_{\R^N}\frac{|\delta_{m}u(x,y)|^2}{|y|^{N+2s}}\ dxdy\right)^{1/2}
\left(\int_{\R^N}\int_{\R^N}\frac{|\delta_{m}v(x,y)|^2}{|y|^{N+2s}}\ dxdy\right)^{1/2}\\
&= \cE_{2m,s}(u,u)^{1/2}\:\cE_{2m,s}(v,v)^{1/2}\leq C\|u\|_{H^s(\R^N)}\|v\|_{H^s(\R^N)}.
\end{align*}
Thus $\cE_{2m,s}$ is a bounded bilinear form in $H^s(\R^N)$, and, by \eqref{smooth},
\begin{align*}
 \cE_{s}(u,v)=\lim_{j\to\infty}\cE_{s}(u_j,v_j)=\lim_{j\to\infty}\cE_{2m,s}(u_j,v_j)=\cE_{2m,s}(u,v).
\end{align*}
\end{proof}

\section{The Fourier symbol}
\label{n:c:sec}

The goal of this section is to show Theorem \ref{main:thm}.  The proof is primarily based on the following.
\begin{thm}\label{const:lemma:2}
	Let $N,m\in\N$, $s\in(0,m)$, and $c_{N,m,s}$ as in $\eqref{cNms:def}$. Then
	\begin{align*}
	\Big(\frac{c_{N,m,s}}{2}\Big)^{-1}=2^{\m} \int_{\R^N} \frac{(1-\cos(y_1))^{\m}}{|y|^{N+2s}}\ dy.
	\end{align*}
\end{thm}

We show first some helpful decompositions and identities.

\begin{lemma}\label{lm}
Let $n\in\N$, $I:=\{1,\ldots,n\}$, and $\rho\in\R$, then 
\begin{align}\label{decom}
\rho^{n-1}=\sum_{k\in I} a_{k,n}\prod_{j\in I\backslash \{k\}}(\rho+j^2),\qquad \text{where }\ \ a_{k,n}:=2\frac{(-1)^{k-n} k^{2n}}{(n+k)!(n-k)!}\qquad \text{ for }k\in I.
\end{align}
Furthermore, if $J\subset\N$ is a finite subset, $k\in \N$, $J(k):=J\cup\{k\}$, and $\rho\in \R$, $\rho\neq -j^2$ for all $j\in J(k)$, then
\begin{align}\label{decom2}
 \frac{1}{\prod_{j\in J(k)}(\rho+j^2)}=\sum_{j\in J(k)}\frac{b_{j,k}}{(\rho+j^2)},\qquad \text{where }\ \ b_{j,k}:=\frac{1}{\prod_{i\in J(k)\backslash\{j\}}(i^2-j^2)}.
\end{align}
\end{lemma}
\begin{proof} 
For $k\in I$ let $f_k(\rho):=\prod_{j\in I\setminus\{k\}}(\rho+j^2)$. The existence of $a_{k,n}\in \R$ for $k\in I$ satisfying the left equality in \eqref{decom} is guaranteed by the fact that $\{f_k\}_{k\in \N}$ form a basis of the space of polynomials of degree less than or equal to $n-1$ (because $f_i(-i^2)\neq 0$ and $f_k(-i^2)=0$ for $k\neq i$, $k,i\in I$).  We now use the \emph{limit method} to show that the coefficients $a_{k,n}$ are as in \eqref{decom}. For $k\in I:=\{1,\ldots,n\}$ we deduce from \eqref{decom} that
\begin{align}\label{step1}
a_{k,n}=\lim_{\rho\to -k^2} \frac{\rho^{n-1}-\sum_{l\in I\backslash\{k\}} a_{l,n}\prod_{j\in I\backslash\{l\}}(\rho+j^2)}{\Pi_{j\in I\backslash\{k\}}(\rho+j^2)}
=\frac{(-1)^{n-1}k^{2n-2}}{\prod_{j\in I\backslash\{k\}}(j^2-k^2)}.
\end{align}
Observe that 
\begin{align*}
(n+k)!=\prod_{j=0}^{n}(k+j)\prod_{j=1}^{k-1}(k-j)\quad\text{ and }\quad (n-k)!=\prod_{j=k+1}^{n}(j-k),
\end{align*}
therefore $(n+k)!(n-k)!=(-1)^{k+1}2k^2\prod_{j\in I\backslash\{k\}}(j+k)(j-k)$, which implies that
\begin{align}\label{step2}
\frac{(-1)^{n-1}k^{2n-2}}{\prod_{j\in I\backslash\{k\}}(j^2-k^2)}=2\frac{(-1)^{k-n} k^{2n}}{(n+k)!(n-k)!},
\end{align}
and thus \eqref{decom} follows from \eqref{step1} and \eqref{step2}.  Equation \eqref{decom2} can be argued similarly: let $J$, $k$, and $\rho$ as stated. The existence of $b_{j,k}$ for $j\in J(k)$ can be argued in a similar way.
Then we can use, as before, the limit method; observe that, for $j\in J(k)$,
\begin{align*}
  b_{j,k}&=\lim_{\rho\to -j^2}(\rho+j^2)(\frac{1}{\prod_{i\in J(k)}(\rho+i^2)}-\sum_{i\in J(k)\backslash\{j\}}\frac{b_{i,k}}{(\rho+i^2)})\\
  &=\frac{1}{\prod_{i\in J(k)\backslash\{j\}}(-j^2+i^2)}-\lim_{\rho\to -j^2}(\rho+j^2)\sum_{i\in J(k)\backslash\{j\}}\frac{b_{i,k}}{(\rho+i^2)}
  =\frac{1}{\prod_{i\in J(k)\backslash\{j\}}(i^2-j^2)},
\end{align*}
and \eqref{decom2} follows.
\end{proof}

\begin{lemma}\label{hc}
 Let $n,j\in\N$ and $j>n$, then
 \begin{align*}
j^{2n-1}=2\sum_{k=1}^n \frac{(-1)^{k-n-1} (j+n)!}{(n+k)!(n-k)!(k^2-j^2)(j-n-1)!}k^{2n}.
 \end{align*}
 \end{lemma}
 \begin{proof}
Let $n,j\in\N$, $j>n$, and $I:=\{1,\ldots,n\}$. By \eqref{decom} with $\rho=-j^2$,
\begin{align*}
(-1)^{n-1}j^{2n-2}&=\sum_{k\in I} \Big(2\frac{(-1)^{k-n} k^{2n}}{(n+k)!(n-k)!}\Big)\prod_{i\in I\backslash \{k\}}(-j^2+i^2)\\
&=\sum_{k\in I} 2\frac{(-1)^{k-n-1} k^{2n}\prod_{i\in I}(j^2-i^2)}{(n+k)!(n-k)!(k^2-j^2)}(-1)^{n-1},
\end{align*}
and the claim follows, since $\prod_{i\in I} j^2-i^2=\prod_{i\in I} (j-i)(j+i)=\frac{(j+n)!}{j(j-n-1)!}$.
\end{proof}

\begin{lemma}\label{const:lemma}
Let $\m\in\N$ and $s\in(0,m)$, then 
\begin{align}
\int_0^{\infty} \frac{\rho^{s-1}}{\prod_{k=1}^{\m}(\rho+k^2)}\ d\rho
=\left\{\begin{aligned}\label{c1}
&2\Gamma(s)\Gamma(1-s)\sum_{k=1}^{\m}\frac{(-1)^{k+1}k^{2s}}{(m-k)!(m+k)!},&& \quad \text{ if }s\not\in\N,\\
&4\sum_{k=1}^m \frac{(-1)^{k-s+1} k^{2s}}{(m+k)!(m-k)!}\ln(k),&&\quad  \text{if $s\in\N$,}
\end{aligned}\right.
\end{align}
\end{lemma}
\begin{proof}
Fix $\m\in\N$, $n\in\{1,\ldots,m\}$, $s\in(n-1,n]\cap (0,m)$, $I:=\{1,\ldots,n\}$, and $J:=\{1,\ldots,m\}\backslash I$. For $k\in I$ let $j\in J(k):=J\cup\{k\}$ and 
\begin{align*}
 a_{k,n}=2\frac{(-1)^{k-n} k^{2n}}{(n+k)!(n-k)!},\qquad b_{j,k}:=\frac{1}{\prod_{i\in J(k)\backslash \{j\}}(i^2-j^2)}.
\end{align*}
Then, by Lemma \ref{lm}, $\rho^{s-1}=\rho^{\sigma-1}\sum_{k\in I} a_{k,n}\Pi_{j\in I\backslash\{k\}}(\rho+j^2)$, where $\sigma:=s-n+1\in(0,1]$, and therefore 
 \begin{align*}
 \int_0^{\infty} \frac{\rho^{s-1}}{\prod_{k=1}^{\m}(\rho+k^2)}\ d\rho
 =\lim_{\eta\to \infty}\sum_{k=1}^n \int_0^{\eta} \frac{a_{k,n}\ \rho^{\sigma-1}}{\prod_{j\in J(k)}(\rho+j^2)}\ d\rho
 =\lim_{\eta\to \infty}\sum_{k\in I}\sum_{j\in J(k)} a_{k,n} b_{j,k}\int_0^{\eta}\frac{\rho^{\sigma-1}}{(\rho+j^2)}\ d\rho.
 \end{align*}
 Let $A_{\eta,\sigma,j}:=\int_0^{\eta}\frac{\rho^{\sigma-1}}{\rho+j^2}\ d\rho$, then 
 \begin{align}
 &\int_0^{\infty} \frac{\rho^{s-1}}{\prod_{k=1}^{\m}(\rho+k^2)}\ d\rho
 =2\lim_{\eta\to \infty}\sum_{k\in I} \frac{(-1)^{k-n} k^{2n}}{(n+k)!(n-k)!}\Big(\frac{A_{\eta,\sigma,k}}{\prod_{i\in J}(i^2-k^2)}+\sum_{j\in J} \frac{A_{\eta,\sigma,j}}{\prod_{i\in J(k)\backslash \{j\}}(i^2-j^2)}\Big)\label{c2}
  \end{align}

Note that, for $j\in J$ and $k\in I$,
\begin{align*}
\prod_{i\in J(k)\backslash \{j\}}(i^2-j^2)=\frac{(-1)^{j-1-n}(k-j)(k+j)(m+j)!(m-j)!(j-n-1)!}{2\ j\ (j+n)!}.
\end{align*}
Therefore, by Lemma \ref{hc},
\begin{align}
 &\sum_{k\in I}\frac{(-1)^{k-n} k^{2n}}{(n+k)!(n-k)!}\sum_{j\in J} \frac{1}{\prod_{i\in J(k)\backslash \{j\}}(i^2-j^2)}\ A_{\eta,\sigma,j}\nonumber\\
 &=\sum_{k\in I}\frac{ k^{2n}}{(n+k)!(n-k)!}\sum_{j\in J} \frac{2(-1)^{k+j-1}\ j\ (j+n)!}{(k-j)(k+j)(m+j)!(m-j)!(j-n-1)!}\ A_{\eta,\sigma,j}\nonumber\\
 &=\sum_{j\in J}\frac{(-1)^{j-n}A_{\eta,\sigma,j}}{(m+j)!(m-j)!}j^{2n}\Big(2j^{1-2n}\sum_{k\in I}\frac{ (-1)^{k-1-n}k^{2n}(j+n)!}{(n+k)!(n-k)!
(k-j)(k+j)(j-n-1)!}\Big)\nonumber\\
  &=\sum_{j\in J}\frac{(-1)^{j-n}A_{\eta,\sigma,j}}{(m+j)!(m-j)!}j^{2n}
  =\sum_{k\in J}\frac{(-1)^{k-n}A_{\eta,\sigma,k}}{(m+k)!(m-k)!}k^{2n}.\label{c2:2}
\end{align}

Furthermore, since $I\cup J=\{1,2,\ldots,m\}$ and  $(n+k)!(n-k)!\prod_{i\in J}(i^2-k^2)=(m+k)!(m-k)!$ we have from \eqref{c2} and \eqref{c2:2} that
\begin{align}\label{c2:3}
 \int_0^{\infty} \frac{\rho^{s-1}}{\prod_{k=1}^{\m}(\rho+k^2)}\ d\rho
 =2\lim_{\eta\to \infty}\sum_{k=1}^m\frac{(-1)^{k-n}k^{2n}}{(m+k)!(m-k)!}A_{\eta,\sigma,k}.
\end{align}

To conclude the proof we consider two cases. In the \emph{first case}, assume that $s\in(n-1,n)$. Then $\sigma\in(0,1)$ and
\begin{align}\label{c2:4}
 A_{\eta,\sigma,k}=\int_0^{\eta}\frac{\rho^{\sigma-1}}{\rho+k^2}\ d\rho=
 k^{2\sigma-2}\int_0^{\eta}\frac{\rho^{\sigma-1}}{\rho+1}\ d\rho.
 \end{align}
Thus, by \eqref{c2:3} and \eqref{c2:4}, we have that
\begin{align*}
 \int_0^{\infty} \frac{\rho^{s-1}}{\prod_{k=1}^{\m}(\rho+k^2)}\ d\rho
 =2\lim_{\eta\to \infty}\int_0^{\eta}\frac{\rho^{\sigma-1}}{\rho+1}\ d\rho\sum_{k=1}^m\frac{(-1)^{k-n}k^{2s}}{(m+k)!(m-k)!}
\end{align*}
and \eqref{c1} follows from \eqref{gamma:id:2}.

For the \emph{second case}, assume that $s=n$. Since $s\in(0,m)$ we have that $n<m$ and $\sigma=1$. Then
\begin{align}\label{c2:5}
 A_{\eta,\sigma,k}=\int_0^{\eta}\frac{1}{\rho+k^2}\ d\rho=\ln(\eta+k^2)-\ln(k^2).
\end{align}

By Lemma \ref{calc:lemma},
\begin{align}\label{zs}
 \sum_{k=1}^m \frac{(-1)^{k}\ k^{2n}}{(m+k)!(m-k)!}=\frac{1}{2(2m)!}\sum_{k=0}^{2m} (-1)^{k}{\binom{2m}{k}}\ (k-m)^{2n}=0,
\end{align}
because $2n<2m$. Let $I_1$ and $I_2$ be the even and odd numbers in $\{1,\ldots, m\}$ respectively. Then, by \eqref{zs}, 
\begin{align*}
  A_1:= \sum_{k\in I_1} \frac{(-1)^{k}\ k^{2n}}{(m+k)!(m-k)!}=\sum_{k\in I_2} \frac{(-1)^{k}\ k^{2n}}{(m+k)!(m-k)!}=:A_2
\end{align*}
and therefore, by the properties of logarithms, 
\begin{equation}\label{c2:6}
\begin{aligned}
\lim_{\eta\to \infty}\sum_{k=1}^m\frac{(-1)^{k}k^{2n}}{(m+k)!(m-k)!}&\ln(\eta+k^2)=\ln\Bigg (\lim_{\eta\to \infty}\frac{\prod_{k\in I_1}(\eta+k^2)^\frac{k^{2n}}{(m+k)!(m-k)!}}{\prod_{k\in I_2}(\eta+k^2)^\frac{k^{2n}}{(m+k)!(m-k)!}}\Bigg )\\
&=\ln\Bigg (\lim_{\eta\to \infty}\frac{\eta^{A_1}\prod_{k\in I_1}(1+\frac{k^2}{\eta})^\frac{k^{2n}}{(m+k)!(m-k)!}}{\eta^{A_2}\prod_{k\in I_2}(1+\frac{k^2}{\eta})^\frac{k^{2n}}{(m+k)!(m-k)!}}\Bigg )
=\ln(1)=0.
\end{aligned}
\end{equation}

From \eqref{c2:3}, \eqref{c2:5}, and \eqref{c2:6}, we conclude that
\begin{align*}
 \int_0^{\infty} \frac{\rho^{s-1}}{\prod_{k=1}^{\m}(\rho+k^2)}\ d\rho
 =2\sum_{k=1}^m \frac{(-1)^{k-n} k^{2n}}{(m+k)!(m-k)!}(-\ln(k^2))
 =4\sum_{k=1}^m \frac{(-1)^{k-s+1} k^{2n}}{(m+k)!(m-k)!}\ln(k),
\end{align*}
as claimed, and this ends the proof.
\end{proof}

\begin{proof}[Proof of Theorem \ref{const:lemma:2}]
By Fubini's theorem, polar coordinates, and \eqref{beta:id},
\begin{align}
\int_{\R^N} \frac{(1-\cos(y_1))^{\m}}{|y|^{N+2s}}\ dy&=2\int_0^{\infty} \frac{(1-\cos(t))^{\m}}{t^{1+2s}}\ dt \int_{\R^{N-1}}\frac{1}{(1+|z|^2)^{\frac{N+2s}{2}}}\ dz\nonumber\\
&=2\frac{2\pi^{\frac{N-1}{2}}}{\Gamma(\frac{N}{2}-\frac{1}{2})}\int_0^{\infty} \frac{(1-\cos(t))^{\m}}{t^{1+2s}}\ dt \int_{0}^{\infty}\frac{\rho^{N-2}}{(1+\rho^2)^{\frac{N+2s}{2}}}\ d\rho\nonumber\\
&=\frac{2 \Gamma(s+\frac{1}{2})\pi^{\frac{N-1}{2}}}{\Gamma(\frac{N}{2}+s)}\int_0^{\infty} \frac{(1-\cos(t))^{\m}}{t^{1+2s}}\ dt,\label{1:m}
\end{align}

Furthermore, since $1-\cos(2a)=1-\cos^2(a)+\sin^2(a)=2\sin^2(a)$ for $a\in\R$,
\begin{align}
\int_{0}^{\infty} \frac{(1-\cos(t))^{\m}}{t^{1+2s}}\ dt= 2^\m\int_0^{\infty} \sin^{2\m}(t/2)t^{-1-2s}\ dt=2^{\m-2s}\int_0^{\infty} \sin^{2\m}(t)t^{-1-2s}\ dt.\label{2:m}
\end{align}
Using a change of variables, \eqref{Gamma:def}, \eqref{Lapl:id}, and Fubini's theorem,
\begin{align}
\Gamma(1+2s)&\int_0^{\infty} \sin^{2\m}(t)t^{-1-2s}\ dt=\int_0^{\infty}\int_0^{\infty} e^{-w} w^{2s} \sin^{2\m}(t) t^{-1-2s}\ dt\ dw\nonumber\\
&=\int_0^{\infty} \int_0^{\infty} e^{-w} \sin^{2\m}(rw) r^{-1-2s}\ dw\ dr=\int_0^{\infty} r^{-2s} \int_0^{\infty} e^{-v/r}\sin^{2\m}(v)\ dv \ dr\nonumber\\
&=\int_0^{\infty} r^{-2-2s} \cL(\sin^{2\m})(r^{-1})\ dr= \int_0^{\infty}\rho^{2s} \cL(\sin^{2\m})(\rho)\ d\rho\nonumber\\
&= (2\m)!\int_0^{\infty}\frac{\rho^{2s-1} }{\prod_{k=1}^{\m}(4k^2+\rho^2)}\ d\rho=(2\m)!2^{2s-2\m} \int_0^{\infty} \frac{\rho^{2s-1} }{\prod_{k=1}^{\m}(k^2+\rho^2)}\ d\rho\nonumber\\
&=(2\m)!2^{2s-2\m-1} \int_0^{\infty} \frac{\rho^{s-1}}{\prod_{k=1}^{\m}(k^2+\rho)}\ d\rho.\label{4:m}
\end{align}
Therefore, by \eqref{1:m}, \eqref{2:m}, \eqref{4:m}, and \eqref{Gammaid}, we obtain
\begin{align*}
\gamma_{N,\m,s}&:=2^{\m} \int_{\R^N} \frac{(1-\cos(y_1))^{\m}}{|y|^{N+2s}}\ dy=2^{\m}\frac{2 \Gamma(s+\frac{1}{2})\pi^{\frac{N-1}{2}}}{\Gamma(\frac{N}{2}+s)}\int_0^{\infty} \frac{(1-\cos(t))^{\m}}{t^{1+2s}}\ dt\\
&=2^{\m}\frac{2 \Gamma(s+\frac{1}{2})\pi^{\frac{N-1}{2}}}{\Gamma(\frac{N}{2}+s)}2^{\m-2s}\int_0^{\infty} \sin^{2\m}(t)t^{-1-2s}\ dt\\
&=\frac{ \Gamma(s+\frac{1}{2})\pi^{\frac{N-1}{2}}}{\Gamma(1+2s)\Gamma(\frac{N}{2}+s)}(2\m)! \int_0^{\infty} \frac{\rho^{s-1}}{\prod_{k=1}^{\m}(k^2+\rho)}\ d\rho\\
&=\frac{(2\m)!\pi^{\frac{N}{2}}}{4^{s}\Gamma(s+1)\Gamma(\frac{N}{2}+s)} \int_0^{\infty} \frac{\rho^{s-1}}{\prod_{k=1}^{\m}(k^2+\rho)}\ d\rho.
\end{align*}
But then $\gamma_{N,m,s}=\frac{2}{c_{N,\m,s}}$, by Lemma \ref{const:lemma}, and this ends the proof.
\end{proof}

\begin{proof}[Proof of Theorem \ref{main:thm}]
Let $u\in\cC^\infty_c(\R^N)$. By the properties of the Fourier transform 
we have, for $\xi\in \R^N\setminus \{0\}$, that
\begin{align*}
\cF(L_{m,s}u)(\xi)&=\frac{c_{N,m,s}}{2}\int_{\R^N} \frac{\cF(\delta_\m u(\cdot,y))(\xi)}{|y|^{N+2s}}\ dy=\frac{c_{N,m,s}}{2}\int_{\R^N}\frac{\delta_{\m} f(0,\xi\cdot y)}{|y|^{N+2s}}\ dy\cF(u)(\xi),
\end{align*}
where $f(t):=\exp(it)$ for $t\in \R$. Moreover, by \eqref{iteration2} and a change of variables,
\begin{align*}
\int_{\R^N}\frac{\delta_{\m} f(0,\xi\cdot y)}{|y|^{N+2s}}\ dy&= 2^{\m}\int_{\R^N}\frac{(1-\cos(\xi\cdot y))^{\m}}{|y|^{N+2s}}\ dy\\
&=2^{\m}|\xi|^{2s}\int_{\R^N}\frac{(1-\cos(\frac{\xi}{|\xi|}\cdot y))^{\m}}{|y|^{N+2s}}\ dy
=|\xi|^{2s}2^{\m} \int_{\R^N} \frac{(1-\cos(y_1))^{\m}}{|y|^{N+2s}}\ dy,
\end{align*}
where the last equality follows by a suitable rotation (\emph{c.f.} \cite[Proposition 3.3]{NPV11}). Then, by Theorem~\ref{const:lemma:2}, 
$\cF(L_{m,s}u)(\xi)=|\xi|^{2s}\cF(u)(\xi)$ for all $\xi\in\R^N\backslash\{0\},$ as claimed. Finally, the identity at $\xi=0$ follows from the fact that
\begin{align*}
\int_{\R^N} L_{m,s}u(x)\;dx\ =\ \frac{c_{N,m,s}}{2}\int_{\R^N}{|y|}^{-N-2s}\int_{\R^N}\delta_m u(x,y)\;dx\;dy\ =\ 0,
\end{align*}
by Fubini's theorem, \eqref{calc}, and a change of variables.
\end{proof}

\end{document}